\documentclass[11pt]{amsart}
\usepackage{etex}

\usepackage{amsmath}
\usepackage{mathrsfs, euscript}
\usepackage{bbm}
\usepackage{amssymb}
\usepackage{amscd}
\usepackage[all,cmtip]{xy}

\usepackage{enumerate}

\newtheorem{theorem}{Theorem}[section]

\newtheorem{lemma}[theorem]{Lemma}
\newtheorem{corollary}[theorem]{Corollary}
\newtheorem{proposition}[theorem]{Proposition}
\newtheorem{remark}[theorem]{Remark}
\newtheorem{definition}[theorem]{Definition}
\newtheorem{deftheorem}[theorem]{Theorem-Definition}

\newcommand{\E}{\mathbb{E}}

\newcommand{\A}{\mathcal{A}}
\newcommand{\B}{\mathcal{B}}

\newcommand{\V}{\mathcal{V}}

\newcommand{\C}{\mathbb{C}}

\newcommand{\p}{{\bf P}}
\newcommand{\ii}{{\bf i}}
\newcommand{\jj}{{\bf j}}

\usepackage[all]{xy}

\begin{document}
\title[De Finetti Theorems]{ General  de Finetti type theorems in noncommutative probability}
\author{Weihua Liu}
\maketitle
\begin{abstract}
We prove general de Finetti type theorems for classical and free independence.
The de Finetti type theorems work for all  non-easy quantum groups, which generalize a recent work of Banica, Curran and Speicher.  
We determine maximal distributional  symmetries which means the corresponding de Finetti type theorem  fails if a sequence of random variables satisfy more symmetry relations other than the maximal one.  
In addition, we define Boolean quantum semigroups in analogous to the easy quantum groups, by universal conditions on matrix coordinate generators and an orthogonal projection.   
Then, we show a general de Finetti type theorem for Boolean independence.
\end{abstract}

\section{Introduction}
The area of distributional symmetries is one of the richest of modern probability theory.  
The most obvious problem in this area is to characterize the class of objects of a given type with a specified symmetric property of their joint distribution. 
For example, de Finetti's fundamental theorem states  that an infinite sequence of random variables, whose joint distribution is invariant under all finite permutations, is conditionally independent and identically distributed.  
Later, in \cite{Freedman}, rotatability and other continuous symmetries were considered by Freedman.  One  can see \cite{Ka1} for more details.

Exchangeability and rotatability are classical symmetries associated with permutation groups and orthogonal groups. 
The quantum analogue of permutation and orthogonal groups were introduced by Wang \cite{Wan2, Wan}. 
They are compact quantum groups in the sense of Woronowicz's matrix pseudogroups \cite{Wo2,Wo}. 
 By using symmetries associated with quantum permutation groups, K\"ostler and Speicher discovered a free analogue of classical de Finetti theorem \cite{KS}:  an infinite sequence of noncommutative random variables are invariant  under quantum permutations is equivalent to the fact that the random variables are identically distributed and free with respect to the conditional expectation onto their tail algebra.
A free analogue of Freedman's work on rotatability was given by Curran in \cite{Cu}.
%it was show that  an infinite sequence of noncommutative random variables are invariant  under quantum orthogonal groups is equivalent to the fact that the random variables satisfy operator-valued free central limit law(semicircular) and free with respect to the conditional expectation onto their tail algebra.

In \cite{BS}, both classical symmetries and quantum symmetries are studied in the \lq\lq easiness\rq\rq formalism. 
Roughly speaking, those  structures are quantum groups associated with tensor categories of partitions. 
 For each $n$, it was shown that there are  exactly six easy groups which are denoted by $S_n$, $O_n$, $B_n$, $H_n$, $B'_n$ and $S'_n$. 
  We will denote the algebras of  continuous functions on these groups by $C_s(n)$, $C_o(n)$, $C_b(n)$, $C_h(n)$, $C_{b'}(n)$ and $C_{s'}(n)$, respectively.
  In the quantum aspects, for each $n$, together with a work of Weber \cite{Weber1}, there are exactly seven easy quantum groups which are denoted by $A_s(n)$, $A_o(n)$, $A_b(n)$, $A_h(n)$, $A_{s'}(n)$, $A_{b'}(n)$ and $A_{b^\#}(n)$.  
  All these algebras are generated by $n^2$ matrix coordinates $\{u_{i,j}|i,j=1,...,n\}$ which satisfy certain relation $R$.  The relations $R$ for $C_*(n)$ and $A_*(n)$ are suitable such that all these algebras are Hopf $C^*$-algebras in the sense of Woronowicz \cite{Wo2}.
 The distributional symmetries associated with Woronowicz's $C^*$-algebras are defined via coactions of quantum groups on noncommutative polynomials in the sense of So{\l}tan \cite{So1}.
  Among these symmetries,   Banica, Curran and Speicher studied de Finetti type theorems for $C_s(n)$, $C_o(n)$, $C_b(n)$, $C_h(n)$ and $A_s(n)$, $A_o(n)$, $A_b(n)$, $A_h(n)$ \cite{BCS}. 
 In short, these symmetries can characterize independence relations which are classical or free.   Furthermore,  it determines some special distributions which are symmetric, shifted central limit or centered central limit laws.   
 One goal of this paper is to study de Finetti type theorems for all  compact quantum groups which are either between $C_s(n)$ and $ C_o(n)$ or between $A_s(n)$ and $ A_o(n)$. 
 
In \cite{RN},  Ryll-Nardzewski showed that  de Finetti theorem holds under the weaker condition of spreadability. 
It shows that  different kinds of distributional symmetries may give the same de Finetti type theorem.  
In this paper,   we study when different kinds of distributional symmetries give the same de Finetti type theorem.
We show that there is no distribution other than  what $C_s(n)$, $C_o(n)$, $C_b(n)$, $C_h(n)$ and $A_s(n)$, $A_o(n)$, $A_b(n)$, $A_h(n)$ can characterize.  
On the other hand, we will show that these distributional symmetries are maximal which means the corresponding de Finetti type theorem  fails if a sequence  of random variables satisfy more symmetries other than the maximal one.

In \cite{Sp, SW}, it was shown that there is a unique non-unital independence, which is called Boolean independence, in noncommutative probability.  
 The study of distributional symmetries for Boolean independence was started  in \cite{Liu}.  
 We constructed a family of quantum semigroups in analogues to Wang's quantum permutation groups and defined their coactions on the joint distribution of a sequence of random variables.
  It  was shown that the distributional symmetries associated with those coactions can be used to characterize  Boolean independence in a proper framework. 
  In a recent work of Hayase \cite{Hay}, following the idea of Banica and Speicher, many distributional symmetries related to Boolean independence  were constructed via the category of interval partitions.  
  By using those distributional symmetries,
   Hayase found de Finetti type theorems for a Boolean analogue of the easy quantum groups.
We will define  quantum semigroups, which are related to Boolean independence, in analogous to the easy quantum groups via some universal conditions.  
The general de Finetti type theorem for Boolean independence is not complete. 
However, with the method introduced in the paper, it is still possible to check de Finetti type theorems for some quantum semigroups other than the Boolean analogue of the easy quantum groups.
 
In the following general de Finetti type theorems, $A\subseteq B$ can be understood  that  all distributional symmetries associated with $A$ are distributional symmetries associated with $B$. 
Unlike the  unified theorem in \cite{BCS}, we divide our de Finetti type theorems in to two cases since we should take account of Gaussian distributions whose support is non-compact.

\begin{theorem}\label{Classical}\normalfont
Let $(\Omega,\Sigma, \mu)$ be a classical probability space and  $(x_i)_{i\in\mathbb{N}}$ be a sequence of real-valued  random variables such that $x_i\in\bigcap\limits_{1\leq p<\infty} L^{p}(\Omega,d\mu)$ for all $i$.
 Let $\A$ be the algebra of all complex-valued random variables and $\phi$ is the classical expectation. Assume that $(x_i)_{i\in\mathbb{N}}$ generate $\A$.
Let $\{E(n)\}_{n\in\mathbb{N}}$ be a sequence of orthogonal Hopf algebras such that $C_s(n)\subseteq E(n)\subseteq C_o(n)$ for each $n\in \mathbb{N}$. If the joint distribution of $(x_i)_{i\in\mathbb{N}}$ is $E(n)$-invariant, then there are a subalgebra $1\subseteq \B\subseteq \A$  and a $\phi$-preserving conditional expectation $\E:\A\rightarrow \B$ such that
\begin{itemize}
\item[1.]  If $E(n)=C_s(n)$ for all $n$, then $(x_i)_{i\in\mathbb{N}}$ are conditionally independent and identically distributed  with respect to $\E$.
\item[2.]  If  $C_s(n) \subseteq E(n) \subseteq C_h(n)$ for all $n$ and there exists a $k$ such that $E(k)\neq C_s(k)$, then $(x_i)_{i\in\mathbb{N}}$ are conditionally independent and have  identical symmetric distribution  with respect to $\E$.
\item[3.]  If  $C_s(n) \subseteq E(n) \subseteq C_b(n)$ for all $n$ and there exists a $k$ such that $E(k)\neq C_s(k)$, then $(x_i)_{i\in\mathbb{N}}$ are conditionally independent and have  identical  shifted-Gaussian distribution  with respect to $\E$.
\item[4.]  If  there exist  $k_1, k_2$ such that $E(k_1)\not\subseteq C_h(k_1)$  and $E(k_2)\not\subseteq C_b(k_2)$, then $(x_i)_{i\in\mathbb{N}}$ are conditionally independent and have  centered Gaussian  distribution  with respect to $\E$.
\end{itemize}
\end{theorem}

The following theorem is for noncommutative probability. We will only consider bounded random variables.
\begin{theorem}\label{Main theorem}\normalfont Let $(\A, \phi)$ be a $W^*$-probability space and $(x_i)_{i\in\mathbb{N}}$ be a sequence of random variables which generate $\A$. 
\begin{itemize}
\item 
Suppose that $\phi$ is faithful.  Let $\{E(n)\}_{n\in\mathbb{N}}$ be a sequence of orthogonal Hopf algebras such that $A_s(n)\subseteq E(n)\subseteq A_o(n)$ for each $n$. If the joint distribution of $(x_i)_{i\in\mathbb{N}}$ is $E(n)$-invariant, then there are a $W^*$-subalgebra $1\subseteq \B\subseteq \A$  and a $\phi$-preserving conditional expectation $\E:\A\rightarrow \B$ such that
\begin{itemize}
\item[1.]  If $E(n)=A_s(n)$ for all $n$, then $(x_i)_{i\in\mathbb{N}}$ are freely independent and identically distributed  with respect to $\E$.
\item[2.]  If  $A_s(n) \subseteq E(n) \subseteq A_h(n)$ for all $n$ and there exists a $k$ such that $E(k)\neq A_s(k)$, then $(x_i)_{i\in\mathbb{N}}$ are freely independent and have  identical symmetric distribution  with respect to $\E$.
\item[3.]  If  $A_s(n) \subseteq E(n) \subseteq A_b(n)$ for all $n$ and there exists a $k$ such that $E(k)\neq A_s(k)$, then $(x_i)_{i\in\mathbb{N}}$ are conditionally independent and have  identical  shifted-semicircular distribution  with respect to $\E$.
\item[4.]  If  there exist  $k_1, k_2$ such that $E(k_1)\not\subseteq A_h(k_1)$  and $E(k_2)\not\subseteq A_b(k_2)$, then $(x_i)_{i\in\mathbb{N}}$ are freely independent and have  centered semicircular  distribution  with respect to $\E$.
\end{itemize}

\item 
Suppose that  $\phi$ is non-degenerated.  Let $\{E(n)\}_{n\in\mathbb{N}}$ be a sequence of orthogonal Boolean quantum semigroups such that $B_s(n)\subseteq E(n)\subseteq B_o(n)$ for each $n$. If the joint distribution of $(x_i)_{i\in\mathbb{N}}$ is $E(n)$-invariant, then there are a $W^*$-subalgebra (does not necessarily contains the unit of $\A$) $\B\subseteq \A$ and a $\phi$-preserving conditional expectation $\E:\A\rightarrow \B$ such that
\begin{itemize}
\item[1.]  If $E(n)=B_s(n)$ for all $n$, then $(x_i)_{i\in\mathbb{N}}$ are Boolean independent and identically distributed  with respect to $\E$.
\item[2.]  If  $B_s(n) \subseteq E(n) \subseteq B_h(n)$ for all $n$ and there exists a $k$ such that $E(k)$ has a quotient algebra $E'(k)$ that $A_s(k)\not\subseteq E'(k)\subseteq A_n(n) $, then $(x_i)_{i\in\mathbb{N}}$ are Boolean independent and have  identical symmetric distribution  with respect to $\E$.
\item[3.]  If  $B_s(n) \subseteq E(n) \subseteq B_b(n)$ for all $n$ and there exists a $k$ such that $E(k)$ has a quotient algebra $E'(k)$ that $A_s(k)\not\subseteq E'(k)\subseteq A_b(n) $, then $(x_i)_{i\in\mathbb{N}}$ are Boolean independent and have  identical  shifted-Bernoulli distribution  with respect to $\E$.
\item[4.]  If  there exist  $k_1, k_2$ such that $E(k_1)$ and $E(k_2)$ have  quotient algebras $E'(k_1)\subseteq A_o(k_1)$ and $ E'(k_2)\subseteq A_o(k_2)$ such that $ E(k_1)\not\subseteq A_h(k_1)$   and $E'(k_2)\not\subseteq A_b(k_2)$, then $(x_i)_{i\in\mathbb{N}}$ are conditionally independent and have  centered Bernoulli distribution  with respect to $\E$.
\end{itemize}

\end{itemize}

\end{theorem}

Besides this introductory section, the rest of this paper is organized as follows.  
In Section 2, we recall some definitions in noncommutative probability and some combinatorial tools. 
In Section 3, we recall the orthogonal Hopf algebras and study their properties.
In Section 4, we define Boolean quantum semigroups in analogous to the easy quantum groups via certain universal conditions.
In Section 5, we give the proof of our main theorems and some applications.

\section{Preliminaries and examples}
In this section, we recall some necessary definitions and notation in noncommutative probability. For further details, see texts   \cite{KS,Liu, NS, VDN}.
\subsection{Noncommutative probability} This part is for noncommutative probability theory and universal independence relations.
\begin{definition} \normalfont
A noncommutative probability space is a pair $(\A,\phi)$, where $\A$ is a unital algebra and $\phi:\A\rightarrow \mathbb{C}$ is a linear functional such that $\phi(1_{\A})=1$. 
$(\A,\phi)$ is a commutative probability space if $\A$ is commutative.
$(\A,\phi)$ is a $C^*$-probability space if $\A$ is a $C^*$-algebra and $\phi$ is a state, i.e. norm one  positive linear functional. 
$(\A,\phi)$ is a $W^*$-probability space if $\A$ is a $W^*$-algebra and $\phi$ is a normal state.
$\phi$ is tracial  if $\phi(xy)=\phi(yx)$ for all $x,y\in \A$.
Assume that $\A$ is a $W^*$-probability space, $\phi$ is faithful if $\phi(xx^*)=0$ if and only if $x=0$,
$\phi$ is  non-degenerated if the GNS representation associated with  $\phi$ is faithful.
The elements in $\A$ are called noncommutative random variables.  
Let $x\in\A$ be a random variable, the distribution of $x$ is a linear functional $\mu_x$ on $\C[X]$ such that $$\mu_x(P)=\phi(P(x))$$ for all $P\in\C[X]$, where $\C[X]$ is the set of complex-coefficient polynomials in one variable.
\end{definition}

\begin{remark}\normalfont
In this paper, the commutative probability space $\A$ is assumed to be set of all complex-valued random variables from $(\Omega,\mu)$, where  $\Omega$ is a classical probability space. It is well know that $\A$ is an algebra and $\phi(\cdot)=\int\cdot d\mu$ defines a linear functional on $\bigcap\limits_{1\leq p<\infty} L^{p}(\Omega)$. In this case, the supports of random variables are not necessarily compact.
When $(\A,\phi)$ is  a $W^*$-probability space where $\phi$ is faithful and tracial,  we can also consider non-compact supported noncommutative random variables.
\end{remark}

\begin{definition}\normalfont Let $I$ be an index set.
 The algebra of noncommutative polynomials in $|I|$ variables, $\C\langle X_i|i\in I\rangle$, is the  linear span of $1$ and noncommutative monomials of the form $X^{k_1}_{i_1}X^{k_2}_{i_2}\cdots X^{k_n}_{i_n}$ with $i_1\neq i_2\neq\cdots \neq i_n\in I$ and all $k_j$'s are positive integers. 
 For convenience, we denote by $\C\langle X_i|i\in I \rangle_0$  the set of noncommutative polynomials without a constant term.
Let $(x_i)_{i\in I}$ be a family of random variables in a noncommutative probability space $(\A,\phi)$.
 Their  joint distribution is a linear functional $\mu:\C\langle X_i|i\in I \rangle\rightarrow \C$ defined by
$$\mu(X^{k_1}_{i_1}X^{k_2}_{i_2}\cdots X^{k_n}_{i_n})=\phi(x^{k_1}_{i_1}x^{k_2}_{i_2}\cdots x^{k_n}_{i_n}),$$
and $\mu(1)=1$.
\end{definition}

\begin{remark}\normalfont
The joint distribution of random variables depends on their order. For example, when $I=\{1,2\}$, then $\mu_{x_1,x_2}$ may not equal $\mu_{x_2,x_1}$. According to our notation, $\mu_{x_1,x_2}(X_1X_2)=\phi(x_1x_2)$, but $\mu_{x_2,x_1}(X_1X_2)=\phi(x_2x_1)$.
\end{remark}

Now, we turn to define  operator-valued probability spaces  which is  a generalization of scalar-valued noncommutative probability spaces.

\begin{definition} \normalfont An operator valued probability space $(\A,\B, \E:\A\rightarrow \B)$ consists of an algebra  $\A$, a subalgebra $\B$ of $\A$ and a $\B-\B$ bimodule linear map $\E:\A\rightarrow \B$ i.e.
 $$\E[b_1ab_2]=b_1\E[a]b_2,  \,\,\, \E[b]=b$$
for all $b_1,b_2,b\in\B$ and $a\in\A$. According to the definition in   \cite{St}, we call $\E$ a conditional expectation from $\A$ to $\B$ if $\E$ is onto, i.e. $\E[\A]=\B$. The elements of $\A$ are called $\B$-valued random variables or random variables for simple.
\end{definition}
%Since the framework for Boolean independence is a non-unital algebra in general, we will not require  our operator valued probability spaces to be unital.

\begin{definition}\normalfont Given an algebra $\B$, we denote by $\B\langle X\rangle$  the algebra which is freely generated by $\B$ and the indeterminant $X$. Let $1_X$ be the identity of $\B\langle X\rangle$, then $\B\langle X\rangle$ is set of linear combinations of the elements in $\B$ and the noncommutative monomials $b_0Xb_1Xb_2\cdots b_{n-1}Xb_n$ where $b_k\in\B\cup \{\C 1_X\}$ and $n\geq 0$. The elements in  $\B\langle X\rangle$ are called $\B$-polynomials. In addition,  $\B\langle X\rangle_0$ denotes the subalgebra of $\B\langle X\rangle$ which does not contain a constant term, i.e. the linear span of the noncommutative monomials $b_0Xb_1Xb_2\cdots b_{n-1}Xb_n$ where $b_k\in\B\cup \{\C 1_X\}$ and $n\geq 1$. 
\end{definition}
Operator-valued independence relations are defined as follows:
\begin{definition}\normalfont
Given an operator valued probability space $(\A,\B,\E:\A\rightarrow \B)$ such that $\A$ and $\B $ are unital. 
\begin{itemize}
\item
Suppose that $\A$ is commutative. A family of unital subalgebras $\{\A_i\supset \B \}_{i\in I}$  are said to be conditionally independent with respect to $\E$ if 
$$\E[a_1\cdots a_n]=\E[a_1]\E[a_2]\cdots \E[a_n],$$
whenever  $a_k\in \A_{i_k}$ and $i_1,...,i_n$ are pairwisely different. A family of random variables $ (x_i)_{i\in I}$ are said to be conditionally independent over $\B$ if the unital subalgebras $\{\A_i\}_{i\in I}$ which are generated by $x_i$ and $\B$ respectively are conditionally independent, or equivalently 
$$\E[p_1(x_{i_1})p_2(x_{i_2})\cdots p_n(x_{i_n})]=\E[p_1(x_{i_1})]\E[p_2(x_{i_2})]\cdots \E[p_n(x_{i_n})],$$
whenever $i_1,..., i_n$ are pairwisely different and $p_1,...,p_n\in \B\langle X\rangle$.\\

\item
 A family of unital subalgebras $\{\A_i\supset \B \}_{i\in I}$  are said to be freely independent with respect to $\E$ if 
$$\E[a_1\cdots a_n]=0,$$
whenever $i_1\neq i_2\neq \cdots\neq i_n$, $a_k\in \A_{i_k}$ and $\E[a_k]=0$ for all $k$. A family of random variables $ (x_i)_{i\in I}$ are said to be freely independent over $\B$, if the unital subalgebras $\{\A_i\}_{i\in I}$ which are generated by $x_i$ and $B$ respectively are freely independent, or equivalently 
$$\E[p_1(x_{i_1})p_2(x_{i_2})\cdots p_n(x_{i_n})]=0,$$
whenever $i_1\neq i_2\neq \cdots\neq i_n$, $p_1,...,p_n\in \B\langle X\rangle$ and $\E[p_k(x_{i_k})]=0$ for all $k$.\\

\item A family of unital subalgebras $\{\A_i\supset \B \}_{i\in I}$  are said to be Boolean independent with respect to $\E$ if 
$$\E[a_1\cdots a_n]=\E[a_1]\E[a_2]\cdots \E[a_n],$$
whenever  $a_k\in \A_{i_k}$ and $i_1\neq i_2\neq\cdots\neq i_n$.    A family of random variables $(x_i)_{i\in I}$ are said to be Boolean independent over $\B$, if the non-unital subalgebras $\{\A_i\}_{i\in I}$ which are generated by $x_i$ and $B$ respectively are Boolean independent, or equivalently 
$$\E[p_1(x_{i_1})p_2(x_{i_2})\cdots p_n(x_{i_n})]=\E[p_1(x_{i_1})]\E[p_2(x_{i_2})]\cdots \E[p_n(x_{i_n})],$$
whenever $i_1\neq i_2\neq\cdots\neq i_n$ and $p_1,...,p_n\in\B\langle X\rangle_0$.
\end{itemize}
\end{definition}

\begin{remark}\normalfont
When $\B=\C$ in the above definition, we get the independence relations for scalar-valued probability.
\end{remark}

\subsection{Partitions and cumulants} In this subsection, we recall some combinatorial tools for classical, free and Boolean independence. See  \cite{AHLV, Lehner, Sp1} for more details.

\begin{definition}\label{partition}\normalfont Let $S$ be an ordered set:
\begin{itemize}
\item[1.] A partition $\pi$ of a set $S$ is a collection of disjoint, nonempty sets $V_1,...,V_r$ whose union is $S$.  $V_1,...,V_r$ are called blocks of $\pi$. 
The collection of all partitions of $S$ will be denoted by $P(S)$.
\item[2.] Given two partitions $\pi$, $\sigma$, we say $\pi\leq \sigma$ if each block of $\pi$ is contained in a block of $\sigma$.
\item[3.]A partition $\pi\in P(S)$ is noncrossing if there is no quadruple $(s_1,s_2,r_1,r_2)$ such that $s_1<r_1<s_2<r_2$, $s_1,s_2\in V$, $r_1,r_2\in W$ and $V,W$ are two different blocks of $\pi$. 
\item[4.]A partition $\pi\in P(S)$ is interval  if there is no triple $(s_1,s_2,r)$ such that $s_1<r<s_2$, $s_1,s_2\in V$, $r\in W$ and $V,W$ are two different blocks of $\pi$. 
\item[5.] Let ${\bf i}=(i_1,...,i_k)$ be a sequence of indices of $I$ and $[k]=\{1,...,k\}$. We denote by ker $\ii$ the element of $P([k])$ whose blocks are the equivalence classes of the relation
$$s\sim t\Leftrightarrow i_s=i_t$$
\end{itemize}
\end{definition}
\begin{remark}\normalfont It is obvious that interval partitions are noncrossing.
\end{remark}

\begin{definition}\normalfont Let $(\A, \B, \E:\A\rightarrow\B)$ be an operator valued probability space:
\begin{itemize}
\item[1.]  A $\B$-functional is a $n$-linear map $\rho:\A^{ n}\rightarrow \B$ such that
$$\rho(b_0a_1b_1,a_2b_2,...,a_nb_n)=b_0\rho(a_1,b_1a_2,...,b_{n-1}a_n)b_n$$
for all $b_0,...,b_n\in \B\cup\{1_A\}$.
\item[2.]For $ k\in\mathbb{N}$, let $\rho^{(k)}$ be a $\B$-functional from $\A^k$ to $\B$.
Suppose that $\A$ is commutative.  Given $\pi\in P(n)$, we can define a $\B$-functional $\rho^{(\pi)}:\A^n \rightarrow \B$ by the formula:
$$\rho^{(\pi)}(a_1,...,a_n)=\prod\limits_{V\in \pi} \rho{(V)}(a_1,...,a_n),$$
where if $V=(i_1<i_2<\cdots<i_s)$ is a block of $\pi$ then 
$$\rho(V)(a_1,...,a_n)=\rho^{(s)}(a_{i_1},...,a_{i_s}). $$
If $\A$ is not commutative, then there is no natural way to define $\rho^{(\pi)}(a_1,...,a_n)$ for $\pi\not\in NC(n)$.
 For $\pi\in NC(n)$,  the $\B$-functional $\rho^{(\pi)}: \A^n\rightarrow \B$ can be defined recursively as follows:
$$\rho^{(\pi)}(a_1,...,a_n)=\rho^{(\pi\setminus V)}(a_1,...,a_l\rho^{(s)}(a_{l+1,...,a_{l+s}}),a_{l+s+1},...,a_n)$$
where $V=(l+1,l+2,...,l+s)$ is an interval block of $\pi$.
\end{itemize}
\end{definition}

\begin{definition}\normalfont
Let $(\A, \B, \E:\A\rightarrow \B)$ be an operator-valued probability space:
\begin{itemize}
\item[1.]  If $\A$ is commutative, then the operator-valued classical cumulants $c^{(n)}_{\E}:\A^n\rightarrow \B$ are defined by the classical moment-cumulant formula:
$$\E[a_{1}\cdots a_{n}]=\sum\limits_{\pi\in P(n)}c_\E^{(\pi)}(a_{1},...,a_{n}),$$
for all $a_1,...,a_n\in \A$. 
\item[2.] The operator-valued  free cumulants $\kappa_\E^{(k)}:\A^n\rightarrow \B$ are defined by the free moment-cumulant formula:
$$\E[a_{1}\cdots a_{n}]=\sum\limits_{\pi\in NC(n)}\kappa_\E^{(\pi)}(a_{1},...,a_{n}),$$
for all $a_1,...,a_n\in \A$. 

\item[3.] The operator-valued  Boolean cumulants $b_\E^{(k)}:\A^n\rightarrow \B$ are defined by the Boolean moment-cumulant formula:
$$\E[a_{1}\cdots a_{n}]=\sum\limits_{\pi\in I(n)}b_\E^{(\pi)}(a_{1},...,a_{n}),$$
for all $a_1,...,a_n\in \A$. 
\end{itemize} 
Note that  all these three types of cumulants can be resolved recursively, e.g. 
$$c_{\E}^{(1)}(a_1)=\E[a_1]$$
and 
$$ c_\E^{(n)}(a_{1},...,a_{n})=\E[a_{1}\cdots a_{n}]-\sum\limits_{\pi\in P(n), \pi\neq 1_n}c_\E^{(\pi)}(a_{1},...,a_{n}),$$
where $c_\E^{(\pi)}(a_{1},...,a_{n})$ depends on $c_\E^{(k)}(a_{1},...,a_{n})$ for $k=1,...,n-1$ if $\pi\neq 1_n$. 
Similarly, we  determine $\kappa_\E^{(n)}$ and $b_\E^{(n)}$ by substituting  $P(n)$ for $NC(n)$ and $I(n)$, respectively.

\end{definition}

See \cite{BCS2} and \cite{Po2}, we have  the following vanishing-cumulant conditions for independence relations.
\begin{theorem}\normalfont
Let $(\A,\B,\E:\A\rightarrow \B)$ be an operator-valued probability space and $(x_i)_{i\in I}$ be a family of random variables in $\A$:
\begin{itemize}
\item[1.] If $\A$ is is commutative,  then $(x_i)_{i\in I}$ are conditionally independent with respect to $\E$ iff
$$c_\E^{(n)}(b_0x_{i_1}b_1,...,x_{i_n}b_n)=0,$$
whenever $i_k\neq i_l$ for some $1\leq k,l \leq n$.
\item[2.] $(x_i)_{i\in I}$ are   free independent with respect to $\E$ iff 
$$\kappa_\E^{(n)}(b_0x_{i_1}b_1,...,x_{i_n}b_n)=0,$$
whenever $i_k\neq i_l$ for some $1\leq k,l \leq n$.
\item[3.]  $(x_i)_{i\in I}$ are  Boolean independent with respect to $\E$ iff
$$b_\E^{(n)}(b_0x_{i_1}b_1,...,x_{i_n}b_n)=0,$$
whenever $i_k\neq i_l$ for some $1\leq k,l \leq n$.
\end{itemize}
\end{theorem}

The following theorem gives  combinatorial characterizations for joint distributions of classical,  free and Boolean independence.  For classical and free independence,  see \cite{BCS2}.  For Boolean independence, see \cite{Hay,Po2}.

\begin{deftheorem}\label{parititionanddistribution}\normalfont
Let $(\A,\B,\E:\A\rightarrow \B)$ be an operator-valued probability space, and $(x_i)_{i\in I}$ be a family of random variables in $\A$:

1. If $\A$ is is commutative,  then the $\B$-valued joint distribution of $(x_i)_{i\in I}$ has the property corresponding to $D$ in the table below  iff for any $\pi\in P(n)$.
$$c_\E^{(\pi)}(b_0x_{i_1}b_1,...,x_{i_n}b_n)=0,$$
unless  $\pi\in D(n)$ and $\pi\leq ker\ii$, where $\ii=(i_1,...,i_n)$.
\begin{table}[h!]
  \begin{center}
    \begin{tabular}{|lcl|}
    \hline
      {\bf Partitions D}& \hspace{3cm} &  {\bf Joint distribution}\\
      \hline
      $P$: All partitions  & & Classical independent\\
      $P_h$: Partitions with even block sizes & & Classical independent and symmetric\\
      $P_b$: Partitions with block size 1 or 2 & & Classical independent and Gaussian\\
      $P_2$: Pair partitions  & & Classical independent and centered Gaussian\\
      \hline
    \end{tabular}
  \end{center}
\end{table}

2. The $\B$-valued joint distribution of $(x_i)_{i\in I}$ has the property corresponding to $D$ in the table below  iff for any $\pi\in P(n)$.
$$\kappa_\E^{(\pi)}(b_0x_{i_1}b_1,...,x_{i_n}b_n)=0,$$
unless  $\pi\in D(n)$ and $\pi\leq ker\ii$, where $\ii=(i_1,...,i_n)$.
\begin{table}[h!]
  \begin{center}
    \begin{tabular}{|lcl|}
    \hline
      {\bf Partitions D}& \hspace{1cm} &  {\bf Joint distribution}\\
      \hline
      $P$: Noncrossing partitions  & & Free independent\\
      $P_h$: Noncrossing Partitions with even block sizes & & Free independent and symmetric\\
      $P_b$: Noncrossing Partitions with block size 1 or 2 & & Free independent and semicircular\\
      $P_2$: Noncrossing Pair partitions  & & Free independent and centered semicircular\\
      \hline
    \end{tabular}
  \end{center}
\end{table}

3. The $\B$-valued joint distribution of $(x_i)_{i\in I}$ has the property corresponding to $D$ in the table below  iff for any $\pi\in P(n)$.
$$b_\E^{(\pi)}(b_0x_{i_1}b_1,...,x_{i_n}b_n)=0,$$
unless  $\pi\in D(n)$ and $\pi\leq ker\ii$, where $\ii=(i_1,...,i_n)$.
\begin{table}[h!]
  \begin{center}

    \label{tab:table1}
    \begin{tabular}{|lcl|}
    \hline
      {\bf Partitions D}& \hspace{2cm} &  {\bf Joint distribution}\\
      \hline
      $I$: Interval partitions  & & Boolean independent\\
      $I_h$: Interval partitions with even block sizes & & Boolean independent and symmetric\\
      $I_b$: Interval partitions with block size 1 or 2 & & Boolean independent and Bernoulli \\
      $I_2$: Interval pair partitions  & & Boolean independent and centered Bernoulli\\
      \hline
    \end{tabular}
  \end{center}
\end{table}
\end{deftheorem}

\begin{remark}\normalfont
Given an operator-valued probability space $(\A,\B,\E)$, a random variable $x$ is symmetrically distributed if and only if 
$$ \E[b_0xb_1\cdots xb_n]=0,$$
whenever $n$ is odd.   
\end{remark}

\begin{remark}\normalfont
Nonzero classical operator-valued centered Gaussian random variables are unbounded.  Let $k$ be a natural number.  It is well known that  $|P_2(2k)|=\frac{(2k)!}{2^k k! }$.  Let $x$ be a nonzero Gaussian random variable from $(\A,\B,\E)$. Then 
we have $a=c^{(2)}_{\E}(x,x)=\E[x^2]\neq 0$ and
$\E[x^{2k}]=\frac{(2k)!}{2^k k! } a^{k}$.  Therefore,
$$ \|x\|\geq \|\E[x^{2k}]\|^{1/2k}=\|\frac{(2k)!}{2^k k! } a^{k}\|^{1/2k}=(\frac{(2k)!}{2^k k! })^{1/2k} \|a\|^{1/2}\rightarrow \infty, \quad k\rightarrow \infty.$$
Therefore, to include Gaussian random variables, we must take account of unbounded random variables.
\end{remark}

\section{ Noncommutative distributional symmetries}
In this section, we will recall  distributional symmetries for classical  and free independence.  See \cite{BS} for more details.  
\begin{definition}\normalfont  An orthogonal  Hopf algebra is a unital $C^*$-algebra $A$ generated by $n^2$ selfadjoint elements  $\{u_{i,j}| i,j=1,...,n\}$, such that the following hold:
\begin{itemize}
\item[1.] The inverse of $u=(u_{i,j})_{i,j=1,....n}\in  M_n(A)$ is the transpose $u^t=(u_{j,i})_{i,j=1,...n}$, i.e. $\sum\limits_{k=1}^n u_{i,k}u_{j,k}=\sum\limits_{k=1}^n u_{k,i}u_{k,j}=\delta_{i,j}1_A$.
\item[2.] $\Delta(u_{i,j})=\sum\limits_{k=1}^nu_{i,k}\otimes u_{k,j}$ determines a $C^*$-unital homomorphism $\Delta:A\rightarrow A\otimes_{min}A$.
\item[3.] $\epsilon(u_{i,j})=\delta_{i,j}$ defines a homomorphism $\epsilon: A\rightarrow \mathbb{C}$.
\item[4.] $S(u_{i,j})=u_{j,i}$ defines a homomorphism $S: A\rightarrow A^{op}$.
\end{itemize}
\end{definition}
 Following  the notion of Wang's free quantum groups in \cite{Wan2,Wan}, one can define universal algebras $A$ generated by $n^2$ noncommutative variables $\{u_{i,j}\}_{i,j=1,...,n}$ which satisfy some relations $R$. 
Moreover, for suitable choices of $R$, we will get Hopf algebras in the sense of Woronowicz \cite{Wo2}. 

In \cite{BS},  Banica and Speicher found the following conditions for defining Hopf orthogonal  algebras.
\begin{definition}\normalfont  A matrix $u=(u_{i,j})_{i,j=1,...,n}\in M_n(A)$ over a $C^*$-algebra $A$ is said to be:
\begin{itemize}
\item Orthogonal, if all entries of $u$ are selfadjoint, and $uu^t=u^tu=1_n$, where $u^t$ the transpose of $u$.	
\item Magic, if it is orthogonal and  its entries are projections.
\item Cubic, if it is orthogonal and  $u_{i,j}u_{i,k}=u_{j,i}u_{k,i}=0$, for $j\neq k$.
\item Bistochastic, if it is orthogonal and  $\sum\limits_{i=1}^n u_{i,j}=\sum\limits_{j=1}^nu_{k,i}=1_{A}$, for all $j,k$.
\item Magic', if it is cubic with the same sum on rows and columns.
\item Bistochastic', if it is orthogonal with the same sum on rows and columns.
\end{itemize}
\end{definition}
The universal algebras associated with the above conditions are defined as follows:
\begin{definition}\normalfont
$A_g(n)$ with $g=o,s,h,b,s',b'$ is the universal $C^*$-algebra generated by the entries of an $n\times n$ matrix which is respectively orthogonal, magic, cubic, bistochastic, magic' and bistochastic'.
$C_g(n)$ with $g=o,s,h,b,s',b'$ is the universal commutative $C^*$-algebra generated by the entries of an $n\times n$ matrix which is respectively orthogonal, magic, cubic, bistochastic, magic' and bistochastic'.
\end{definition}
Especially, for each $n$, $A_s(n)$ and $A_o(n)$ are Wang's quantum permutation group and quantum orthogonal group, respectively \cite{Wan, Wan2}.  
$C_g(n)$ can be considered as the abelianization of $A_g(n)$ for $g=o,s,h,b,s',b'$. 
It should be mentioned here that there are 7 easy quantum groups in total, see \cite{Weber1}.

Given two orthogonal Hopf algebras $(A,u)$ and $(B,v)$ such that $A$ is generated by $u=\{u_{i,j}|i,j=1,...n\}$ and $B$ is generated by $v=\{v_{i,j}|i,j=1,...n\}$.  We denote by $(A,u)\rightarrow (B,v)$ if there exists a $C^*$-homomorphism $\eta$ from $A$ to $B$ such that $\eta(u_{i,j})=v_{i,j}.$  In other words, $(A,u)\rightarrow (B,v)$ implies that $B$ is a quotient $C^*$-algebra of $A$.   Then, for each $n$,  we have the following diagrams:
\begin{displaymath}
    \xymatrix{  A_o(n)\ar[r]\ar[d]& A_{b'}(n) \ar[r]\ar[d] & A_b(n) \ar[d] \\  
                         A_h(n)\ar[r]& A_{s'}(n) \ar[r]  & A_s(n)\\                  }
\end{displaymath}
and 
\begin{displaymath}
    \xymatrix{  C_o(n)\ar[r]\ar[d]& C_{b'}(n) \ar[r]\ar[d] & C_b(n) \ar[d] \\  
                         C_h(n)\ar[r]& C_{s'}(n) \ar[r]  & C_s(n)\\                  }
\end{displaymath}
and 
$$A_g(n)\rightarrow C_g(n),$$
for $g=o,s,h,b,s',b'.$  

For convenience, we will denote by $B\subseteq A$  the relation $(A,u)\rightarrow (B,v)$.

\begin{proposition}\label{3.5}\normalfont Let $E(n)$ be an orthogonal Hopf algebra generated by $n^2$ selfadjoint elements $\{u_{i,j}\}_{i,j=1,...,n}$.  Then, we have the following.
\begin{itemize}
\item[1.] If  $E(n)\not\subseteq A_h(n)$, then there exists an index  $j$ such that $\sum\limits_{k=1}^n u_{k,j}^4\neq 1_{E(n)}$.
\item[2.] If  $E(n)\not\subseteq A_b(n)$, then there exists an index  $j$ such that $\sum\limits_{k=1}^n u_{k,j}\neq 1_{E(n)}$.
\end{itemize} 
\end{proposition}
\begin{proof}
1. Suppose $\sum\limits_{k=1}^n u_{k,i}^4=1_{E(n)}$, for all $i$.
Since $\sum\limits_{k=1}^n u_{k,i}^2=1_{E(n)}$ and $u_{k,i}^4\leq u_{k,i}^2 $, we have $$u_{k,i}^4=u_{k,i}^2.$$  $(u^2_{i,j})_{i,j=1,...,n}$ is a matrix of orthogonal projections with sum 1 on rows and columns.  Therefore, 
$$u^2_{i,j}u^2_{i,k}=u^2_{j,i}u^2_{k,i}=0,$$
for $j\neq k.$ Since $u_{i,j}$ and $u_{i,k}$ are selfadjoint, we have 
$$u_{i,j}u_{i,k}=u_{j,i}u_{k,i}=0.$$
It implies that $E(n)$ is a quotient algebra of $A_h(n)$, which is a contradiction.

2. Suppose $\sum\limits_{k=1}^n u_{k,i}=1_{E(n)}$, for all $i$. Then, for each $i$, we have
$$\sum\limits_{l=1}^n u_{i,l}=\sum\limits_{l=1}^n \sum\limits_{k=1}^nu_{i,l} u_{k,l}=\sum\limits_{k=1}^n \sum\limits_{l=1}^n u_{i,l}u_{k,l}=\sum\limits_{k=1}^n\delta_{i,k}1_{E(n)}=1_{E(n).}$$
Therefore, $E(n)$ is a quotient algebra of $A_b(n)$, which is a contradiction.
\end{proof}

\begin{proposition}\label{3.4}\normalfont
 Let $E(n)$ be an orthogonal Hopf algebra generated by $n^2$ selfadjoint elements $\{u_{i,j}\}_{i,j=1,...,n}$ such that $A_s(n)\subseteq E(n)\subseteq A_o(n)$. Then, the following hold:
 \begin{itemize}
 \item[1.] If $E(n)\subseteq A_h(n)$ and $E(n)\subseteq A_b(n)$, then $E(n)=A_s(n)$.
 \item[2.] If $E(n)\not\subseteq A_h(n)$ and $E(n)\subseteq A_b(n)$, then $\exists\, i'$ such that
 $$\sum\limits_{k=1}^n u^m_{k,i'}\neq 1,$$
 for all $m>2$.
  \item[3.] If $E(n)\not\subseteq A_b(n)$ and $E(n)\subseteq A_h(n)$, then $\exists\, i'$ such that
 $$\sum\limits_{k=1}^n u^m_{k,i'}\neq 1,$$
 for all odd numbers $m$. 
 
 \item[4.] If $E(n)\not\subseteq A_h(n)$ and $E(n)\not\subseteq A_b(n)$, then $\exists\, i'_1, i'_2$ such that
 $$\sum\limits_{k=1}^n u^m_{k, i_1'}\neq 1,$$
 for all $m> 2$, and 
 $$\sum\limits_{k=1}^n u_{k, i'_2}\neq 1.$$
 \end{itemize}
\end{proposition}
\begin{proof}
It is obvious that  $\|u_{i,j}\|\leq 1$
for all $i,j=1,...n$. 

1.  By assumption, we have $$\sum\limits_{k=1}^n u_{i,k}=1_{E(n)}$$
and
$$u_{i,j}u_{i,k}=0$$
for $j\neq k$. Therefore,  
$$u_{i,j}=u_{i,j}\sum\limits_{k=1}^n u_{i,k}=u^2_{i,j}$$
for all $i,j$.
It implies that $E(n)$ is a quotient algebra of $A_s(n)$, so $E(n)=A_s(n).$\\

2. By Proposition \ref{3.5},  there exists $i'$ such that 
$$ \sum\limits_{k=1}^n u^4_{k, i'}\neq 1.$$ Therefore, there exists $k'$ such that 
$$u^4_{k',i'}<u^2_{k',i'}$$
which implies that the spectrum of $u_{k',i'}$ contains $a$ such that $-1<a<1$. Therefore, 
$$u^m_{k',i'}<u^2_{k',i'}$$ 
for all natural number $m>2$. Hence, we have 
$$\sum\limits_{k=1}^n u^m_{k,i'}< 1_{E(n)},$$
for $m>2$.

3. According to Proposition \ref{3.5},  there exists $i'$ such that 
$$ \sum\limits_{k=1}^n u_{k, i'}\neq 1.$$ Therefore, there exists $k'$ such that 
$u_{k',i'}$ is not an orthogonal projection which implies that 
$$u^{2m+1}_{k',i'}<u^{2m}_{k',i'}.$$ Thus, we have
$$\sum\limits_{k=1}^n u^{2m+1}_{k,i'}<\sum\limits_{k=1}^n u^{2m}_{k,i'}=1_{E(n)} ,$$

4. Combine  Case 2 and 3,  the proof is complete.
\end{proof}

Following the proof above, we have following properties for abelian orthogonal Hopf algebras.
\begin{corollary}\normalfont Let $E(n)$ be an orthogonal Hopf algebra generated by $n^2$ selfadjoint elements $\{u_{i,j}\}_{i,j=1,...,n}$ such that $C_s(n)\subseteq E(n)\subseteq C_o(n)$. Then, the following hold:
 \begin{itemize}
 \item[1.] If $E(n)\subseteq C_h(n)$ and $E(n)\subseteq C_b(n)$, then $E(n)=C_s(n)$.
 \item[2.] If $E(n)\not\subseteq C_h(n)$ and $E(n)\subseteq C_b(n)$, then there exists $ i'$ such that
 $\sum\limits_{k=1}^n u^m_{k,i'}\neq 1,$
 for all $m>2$.
  \item[3.] If $E(n)\not\subseteq C_b(n)$ and $E(n)\subseteq C_h(n)$, then there exists $i'$ such that
 $\sum\limits_{k=1}^n u^m_{k,i'}\neq 1,$
 for all odd numbers $m$. 
 
 \item[4.] If $E(n)\not\subseteq C_h(n)$ and $E(n)\not\subseteq C_b(n)$, then there exist $ i'_1, i‘_2$ such that
 $\sum\limits_{k=1}^n u^m_{k, i_1'}\neq 1,$
 for all $m> 2$, and 
 $\sum\limits_{k=1}^n u_{k, i'_2}\neq 1.$
 \end{itemize}
\end{corollary}

Now, we turn to define noncommutative distributional symmetries by maps of quantum family in the sense of  So{\l}tan \cite{So}.

\begin{definition}\normalfont
Let $(A,\Delta)$ be a quantum  group and $\V$ be a unital algebra. By a (right) coaction of the quantum group $A$ on $\V$, we mean a unital homomorphism $\alpha:\V\rightarrow\V\otimes \A$ such that
$$(\alpha\otimes id_A)\alpha=(id\otimes \Delta)\alpha.$$
\end{definition}

\begin{definition}\normalfont Given an orthogonal Hopf algebra $E(n)$ generated by $\{u_{i,j}\}_{i,j=1,...n}$, we have a natural coaction $\alpha_n$ of $E(n)$ on $\mathbb{C}\langle X_1,...,X_n \rangle$ such that 
$$\alpha_n: \mathbb{C}\langle X_1,...,X_n \rangle\rightarrow \mathbb{C}\langle X_1,...,X_n \rangle\otimes E(n) $$
is an algebraic homomorphism defined via
$\alpha_n(X_i)=\sum\limits_{k=1}^n X_k\otimes u_{k,i}$ for all $i=1,...,n$.
\end{definition}

\begin{definition}\normalfont Given a probability space $(\A,\phi)$, a sequence of random variables $(x_1,...,x_n)$ of $\A$ and an orthogonal Hopf algebra $E(n)$ generated by $\{u_{i,j}\}_{i,j=1,...n}$. We  say that the joint distribution $\mu_{x_1,...,x_n}$ of $x_1,...,x_n$ is $E(n)$-invariant if 
$$ \mu_{x_1,...,x_n}(p)1_{E(n)}=\mu_{x_1,...,x_n}\otimes id_{E(n)}(\alpha_n(p)),$$
for all $p\in \C\langle X_1,...,X_n\rangle$.
\end{definition}

\begin{proposition}\label{Implication}\normalfont Given a probability space $(\A,\phi)$ and a sequence of random variables $(x_1,...,x_n)$ of $\A$.  $E_1(n)$ and $E_2(n)$ are two orthogonal Hopf algebras such that $E_1(n)\subseteq E_2(n)$. Then, $(x_1,...,x_n)$ is  $E_1(n)$-invariant if $E_2(n)$-invariant.
\end{proposition}
\begin{proof}
Let $\{u^{(l)}_{i,j}\}_{i,j=1,...,n}$ be generators of $E_l(n)$ for $l=1,2.$  Since $E_1(n)\subseteq E_2(n)$, there exists a $C^*$-homomorphism $\Phi:E_2(n)\rightarrow E_1(n)$ such that 
$$\Phi(u^{(2)}_{i,j})=u^{(1)}_{i,j}$$
for all $i,j$.   $(x_1,...,x_n)$ is   $E_2(n)$-invariant is equivalent to that
$$ \mu_{x_1,...,x_n}(X_{\ii})1_{E_2(n)}=\sum\limits_{\jj\in[n]^k}\mu_{x_1,...,x_n}(X_{\jj})\otimes u^{(2)}_{\ii,\jj},$$
for all monomials  $X_{i_1}\cdots X_{i_k}\in \C\langle X_1,...,X_n\rangle$. Apply $\Phi$ on both sides of the above equation, we get
$$ \mu_{x_1,...,x_n}(X_{\ii})1_{E_1(n)}=\sum\limits_{\jj\in[n]^k}\mu_{x_1,...,x_n}(X_{\jj})\otimes u^{(1)}_{\ii,\jj},$$
which implies that $(x_1,...,x_n)$ is   $E_1(n)$-invariant.
\end{proof}

Given an orthogonal Hopf algebra $E(n)$ generated by $\{u_{i,j}\}_{i,j=1,...,n}$. Then, for $k\in\mathbb{N}$, $E(n)$ can be considered as an orthogonal Hopf algebra $E(n,k)$ generated by $\{v_{i,j}\}_{i,j=1,...,n+k}$ such that 
$$ v_{i,j}=
\left\{\begin{array}{ll}
u_{i,j} & \text{if}\,\, i,j \leq  n\\
\delta_{i,j}1_{E(n)} &\text{otherwise}
\end{array}\right..
$$
We will call $E(n,k)$ the $k$-th extension of $E(n)$. To study de Finetti type theorems for all orthogonal Hopf algebras $E(n)$, we need to extend $E(n)$-invariance condition on $n$ random variables to an invariance condition on infinitely many random variables.
\begin{definition}\normalfont\label{stationary}
 Given a probability space $(\A,\phi)$, a sequence of random variables $(x_i)_{i\in\mathbb{N}}$ of $\A$ and an orthogonal Hopf algebra $E(n)$ generated by $\{u_{i,j}\}_{i,j=1,...n}$. We  say that the joint distribution $\mu$ of $(x_i)_{i\in\mathbb{N}}$ is $E(n)$-invariant if 
the joint distribution  of $(x_1,...,x_{n+k}) $ is $E(n,k)$-invariant for all $k\in\mathbb{N}$. 
\end{definition}

\section{Boolean Quantum semigroups in analogous to easy quantum groups}
Inspired by the previous work in \cite{Liu}, we will define distributional symmetries for Boolean independent random variables via quantum semigroups.  
For any $C^*$-algebras $A$ and $B$, the set of  morphisms Mor$(A,B)$ consists of all $C^*$-algebra homomorphisms acting from $A$ to $M(B)$, where $M(B)$ is the multiplier algebra of $B$, such that $\phi(A)B$ is dense in $B$.
If $A$ and $B$ are unital $C^*$-algebras, then all unital $C^*$-homomorphisms from $A$ to $B$ are in Mor(A,B). 
In \cite{So}, the morphisms in the category of quantum semigroups are defined as follows.

\begin{definition}\normalfont
By a quantum semigroup we mean a $C^*$-algebra $\A$ endowed with an additional structure  described by a morphism $\Delta\in Mor(\A,\A\otimes\A)$ such that
     $$(\Delta\otimes id_{\A})\Delta=(id_{\A}\otimes\Delta)\Delta.$$
\end{definition}
The quantum semigroups for Boolean independence are unital universal $C^*$-algebras generated by  an orthogonal projection $\p$ and entries of $n\times n$ matrices $u=(u_{i,j})_{i,j=1,...,n}$  which satisfy certain relation $R$ related to $\p$.

\begin{definition}\label{p-condition}\normalfont
Let  $u=(u_{i,j})_{i,j=1,...,n}\in M_n(\A)$ be an $n\times n$ matrix over a $C^*$-algebra $\A$ and $\p$ be an orthogonal projection in $\A$.  The pair $(u,\p)$ is said to be:
\begin{enumerate}
\item  $\p$-orthogonal, if all entries of $u$ are selfadjoint, and $uu^t(1_n\otimes\p)=u^tu(1_n\otimes\p)=1_n\otimes\p$ i.e. $\sum\limits_{k=1}^n u_{i,k}u_{j,k}\p=\sum\limits_{k=1}^n u_{k,i}u_{k,j}\p=\delta_{i,j}\p$.
\item $\p$-magic, if it is $\p$-orthogonal and   the entries of $u$ are projections.
\item $\p$-cubic, if it is $\p$-orthogonal and  $u_{i,j}u_{i,k}\p=u_{j,i}u_{j,k}\p=0$, for $j\neq k$.
\item $\p$-bistochastic, if it is $\p$-orthogonal, and  $\sum\limits_{j=1}^n u_{i,j}\p=\sum\limits_{k=1}^nu_{k,i}\p=\p$, for all $i$.
\item $\p$-', if $\sum\limits_{j=1}^n u_{i,j}\p=\sum\limits_{k=1}^nu_{k,i}\p$, for all $i$.
\item $\p$-magic', if it is $\p$-cubic and $\p$-'.
\item $\p$-bistochastic', if it is $\p$-orthogonal and $\p$-'.
\end{enumerate}
\end{definition}

Unlike universal conditions for quantum groups, these conditions cannot define universal $C^*$-algebras since they cannot ensure that $u_{i,j}$'s are bounded. Therefore, we need an additional condition to control the norms of $u_{i,j}'$s.  We say $(u_{i,j})_{i,=1,...n}$ is norm $\leq 1$ if the norm $\|(u_{i,j})_{i,j=1,...,n}\|\leq 1.$

\begin{definition} \label{Boolean quantum semigroup}\normalfont
Let $B_g(n)$ with $g=o,s,h,b,s',b'$ be the unital universal $C^*$-algebra generated by the entries of a norm $\leq1$ matrix $u=(u_{i,j})_{i,=1,...n}$ and an orthogonal projection $\p$ when the pair $(u,\p)$ is  $\p$-orthogonal, $\p$-magic, $\p$-cubic, $\p$-bistochastic, $\p$-magic' and $\p$-bistochastic', respectively.
\end{definition}

On the  $C^*$-algebra $B_g(n)$ with $g=o,s,h,b,s',b'$, we can always define a unital $C^*$-homomorphism $$\Delta:B_g(n)\rightarrow B_g(n)\otimes B_g(n)$$ on their generators by the following equations:
$$\Delta u_{i,j}=\sum\limits_{k=1}^nu_{i,k}\otimes u_{k,j}$$
and $$\Delta{\bf P}={\bf P}\otimes {\bf P},\,\,\,\, \Delta I=I\otimes I.$$

To show the coproduct $\Delta$ is well defined, we need to show that the $(\Delta u_{i,j})_{i,j=1,...,n}$ and $\p\otimes\p$ satisfy the universal conditions as  $( u_{i,j})_{i,j=1,...,n}$ and $\p$ do.  We check them below.

\noindent{\bf Norm condition:}  If $\|(u_{i,j})_{i,j=1,...n}\|\leq 1$, then we have
$$\|(\Delta u_{i,j})_{i,j=1,...n}\|=\|(\sum\limits_{k=1}^n u_{i,k}\otimes u_{k,j})_{i,j=1,...n}\|=\|(u_{i,j}\otimes 1_n)_{i,j=1,...n}(1_n\otimes u_{i,j})_{i,j=1,...n}\|\leq \|( u_{i,j})_{i,j=1,...n}\|^2\leq1.$$
{\bf $\p$-orthogonal:} If $\sum\limits_{k=1}^nu_{i,k}u_{j,k}\p=\sum\limits_{k=1}^nu_{k,i}u_{k,j}\p=\delta_{i,j}\p$, then

\begin{align*}
\sum\limits_{k=1}^n\Delta u_{i,k}\Delta u_{j,k}\Delta\p 
&= \sum\limits_{k=1}^n (\sum\limits_{l=1}^n u_{i,l}\otimes u_{l,k})(\sum\limits_{m=1}^n u_{j,m}\otimes u_{m,k})(\p\otimes \p)\\
&= \sum\limits_{k=1}^n \sum\limits_{l=1}^n \sum\limits_{m=1}^n u_{i,l}u_{j,m}\p\otimes u_{l,k}u_{m,k}\p\\
&= \sum\limits_{l=1}^n \sum\limits_{m=1}^n u_{i,l}u_{j,m}\p\otimes \delta_{m,l}\p\\
&= \sum\limits_{l=1}^n u_{i,l}u_{j,l}\p\otimes \p\\
&= \delta_{i,j}\p\otimes \p.\\
\end{align*}
Similarly, we have $\sum\limits_{k=1}^n\Delta u_{k,i}\Delta u_{k,j}\Delta\p=\delta_{i,j}\p\otimes \p.$\\
{\bf $\p$-cubic:}  Assume that  $u_{i,j}u_{i,k}\p=u_{j,i}u_{j,k}\p=0$, for $j\neq k$, then we have 
\begin{align*}
\Delta u_{i,j}\Delta u_{i,k}\Delta\p 
&= \sum\limits_{l,m=1}^n u_{i,l} u_{i,m}\p\otimes  u_{l,j}u_{m,k}\p\\
&= \sum\limits_{l=1}^n u_{i,l} u_{i,l}\p\otimes  u_{l,j}u_{l,k}\p\\
&=0,
\end{align*}
whenever $j\neq k$.  Similarly, we have  
$$\Delta u_{j,i} \Delta u_{j,k} \Delta\p=0,$$
whenever $j\neq k$.\\
{\bf $\p$-bistochastic:} If $\sum\limits_{j=1}^n u_{i,j}\p=\sum\limits_{j=1}^nu_{j,i}\p=\p$, for all $j=1,...,n$. Then we have
\begin{align*}
\sum\limits_{j=1}^n \Delta u_{i,j}\Delta\p \
= \sum\limits_{j=1}^n\sum\limits_{k=1}^n u_{i,k}\p\otimes  u_{k,j}\p
= \sum\limits_{j=1}^n u_{i,j}\p\otimes \p
=\p\otimes \p.
\end{align*}
Similarly, we  have  
$\sum\limits_{j=1}^n\Delta u_{j,i}\Delta\p=\p\otimes \p, $
for all $j.$\\
{\bf $\p$' -condition: } Let $r=\sum\limits_{j=1}^n u_{i,j}\p=\sum\limits_{j=1}^nu_{j,i}\p$, for $j\neq k$.  Then we have
$$
\sum\limits_{j=1}^n\Delta u_{i,j}\Delta\p 
=\sum\limits_{j,l=1}^n u_{i,l} \p\otimes  u_{l,j}\p
= \sum\limits_{l=1}^n u_{i,l}\p\otimes  r
=r\otimes r,
$$ for all $j$.
Similarly, we  have  
$\sum\limits_{j=1}^n\Delta u_{j,i}\Delta\p=r\otimes r $
for all $j.$

Therefore, $\Delta$ is a well defined $C^*$-homomorphism and $(B_g(n),\Delta)$ is a quantum semigroup for $g=o,s,h,b,s',b'$. As the relations for easy quantum groups, we have the following diagram for Boolean quantum semigroups:
\begin{displaymath}
    \xymatrix{  B_o(n)\ar[r]\ar[d]& B_{b'}(n) \ar[r]\ar[d] & B_b(n) \ar[d] \\  
                         B_h(n)\ar[r]& B_{s'}(n) \ar[r]  & B_s(n)\\     }
\end{displaymath}
Furthermore, for $g=o,s,h,b,s',b'$, $A_g(n)$ is a quotient algebra of $B_g(n)$ by requiring $\p$ to be the unit of the algebra.  Therefore, we have
$$ C_g(n)\subseteq A_g(n)\subseteq B_g(n) $$
for $g=o,s,h,b,s',b'$.\\
In addition, the algebras $B_g(n)$ generated by the generators of $B_g(n)$  with $g=o,s,h,b$  are quotient algebras of Hayase's Hopf algebras $C(G_n^{I_2}), C(G_n^{I}),C(G_n^{I_h}),C(G_n^{I_b})$ in \cite{Hay}, respectively. 
Indeed,  $B_g(n)$  with $g=o,s,h,b$ satisfy Hayase's universal conditions for $C(G_n^{I_2}), C(G_n^{I}),C(G_n^{I_h}),C(G_n^{I_b})$, respectively. 
To check the vanishing conditions, we apply the following notation for convenience:
Given $\pi_1\in I(k_1)$ and $\pi_2\in I(k_2)$,  $\pi=\pi_1\pi_2\in I(k_1+k_2) $ denotes the concatenation of $\pi_1$ and $\pi_2$. Given $\jj_1=(j_1,...,j_{k_1})\in [n]^{k_1} $ and $\jj_2=(j'_1,...,j'_{k_2})\in[n]^{k_2}$, $\jj=\jj_1\jj_2=(j_1,...,j_{k_1},j'_1,...,j'_{k_2})\in[n]^{k_1+k_2}$. According to Definition \ref{partition}, the following lemma is obvious.
\begin{lemma} \normalfont
Let $\pi\in I(k_1+k_2)$ such that $\pi=\pi_1\pi_2$ for some $\pi_1\in I(k_1)$ and $\pi_2\in P(k_2)$. Let $ \jj=\jj_1+\jj_2$ such that $\jj_1\in [n]^{k_1} $ and $\jj_2\in[n]^{k_2}$. Then, $\pi\leq ker\jj$ iff $\pi_i\leq \ker \jj_i$ for $i=1,2$.
\end{lemma}
Therefore, we have the following property.
\begin{lemma}\label{concatenation}\normalfont Given $\pi_1\in I(k_1)$, $\pi_2\in P(k_2)$ and $ \jj=\jj_1+\jj_2$ such that $\jj_1\in [n]^{k_1} $ and $\jj_2\in[n]^{k_2}$. If
$$ \sum\limits_{\ii_i\in[n]^{k_i},\pi_i\leq \ker\ii_i} u_{\ii_i,\jj_i}\p=
\left\{\begin{array}{ll}
\p & \text{if}\,\, \pi\leq\ker\jj_i\\
0 &\text{otherwise}
\end{array}\right.
$$
for $i=1,2$.  Then, we have 
$$ \sum\limits_{\ii\in[n]^{k_1+k_2},\pi_1\pi_2\leq \ker\ii} u_{\ii,\jj}\p=
\left\{\begin{array}{ll}
\p & \text{if}\,\, \pi\leq\ker\jj\\
0 &\text{otherwise}
\end{array}\right.
$$
\end{lemma}
\begin{proof}
By a direct computation, we have:
$$ 
\sum\limits_{\substack{\ii\in[n]^{k_1+k_2}\\ \pi_1\pi_2\leq \ker\ii}} u_{\ii,\jj}\p
=\sum\limits_{\substack{\ii_1\in[n]^{k_1}\\ \pi_1\leq \ker\ii_1}}\sum\limits_{\substack{\ii_2\in[n]^{k_2}\\\pi_2\leq \ker\ii_2}}  u_{\ii_1,\jj_1}u_{\ii_2,\jj_2}\p
=\left\{\begin{array}{ll}
\sum\limits_{\substack{\ii_1\in[n]^{k_1}\\ \pi_1\leq \ker\ii_1}} u_{\ii_1,\jj_1}\p& \text{if}\,\, \pi_1\leq\ker\jj_2\\
0 &\text{otherwise,}
\end{array}\right.\\
$$
Therefore,
$$\sum\limits_{\substack{\ii\in[n]^{k_1+k_2}\\ \pi_1\pi_2\leq \ker\ii}} u_{\ii,\jj}\p=\left\{\begin{array}{ll}
\p&\,\, \pi_1\leq\ker\jj_1\,\text{and}\, \pi_2\leq\ker\jj_2\\
0 &\text{otherwise}
\end{array}\right.
$$
which completes the proof.
\end{proof}

Now, we turn to check the following  vanishing conditions.
\begin{lemma}\label{vanishing}\normalfont
Let $u_{i,j}$'s and $\p$ be the standard generators of $B_o(n)$,$B_s(n)$,$B_h(n)$,$B_b(n)$. Then, we have 
$$ \sum\limits_{\ii\in[n]^k,\pi\leq \ker\ii} u_{\ii,\jj}\p=
\left\{\begin{array}{ll}
\p & \text{if}\,\, \pi\leq\ker\jj\\
0 &\text{otherwise}
\end{array}\right.
$$
for $\pi\in I_2(k),I(k),I_h(k),I_b(k)$, respectively.
\end{lemma}
\begin{proof}
1. For $B_o(n)$,  $k=2$.  The identity holds by the definition of $B_o(n)$.  Since all partitions in $I_2(n)$ are concatenations of pair partitions by Lemma \ref{concatenation}, the equation holds.

2. For $B_b(n)$,  the identity holds by the definition of $B_b(n)$ when $\pi$ is a single partition or a partition. 
Since all partitions in $I_b(n)$ are concatenations of  single partitions and pair partitions,
by Lemma \ref{concatenation}, the equation holds.

3. For $B_h(n)$ we  need to check $\pi=1_{2m}\in I_h(2m)$ for all $m\in\mathbb{N}$. It follows that
$$\sum\limits_{\substack{ \ii\in[n]^{2m}\\ \pi\leq \text{ker}\,\ii }}u_{\ii,\jj}\p=\sum\limits_{i=1}^nu_{i,j_1}\cdots u_{i,j_{2m}}\p.$$
It equals zero if $j_l\neq j_{l+1}$ for some $l$, otherwise
$$ \sum\limits_{i=1}^nu_{i,j_1}\cdots u_{i,j_{2m}}\p=\sum\limits_{i=1}^nu_{i,j_1}^{2m}\p=\sum\limits_{i=1}^nu_{i,j_1}^{2m-2}\sum\limits_{l=1}^nu^2_{l,j_1}\p=\sum\limits_{i=1}^nu_{i,j_1}^{2m-2}\p=\cdots=\sum\limits_{l=1}^nu^2_{l,j_1}\p=\p.$$
Since all partitions in $I_b(n)$ are concatenations of  blocks of even length, by Lemma \ref{concatenation}, the equation holds.

4. For $B_s(n)$, we  need to check $\pi=1_{m}\in I(m)$, for all $m\in\mathbb{N}$. It follows that
$$\sum\limits_{\substack{ \ii\in[n]^m\\ \pi\leq \text{ker}\,\ii }}u_{\ii,\jj}\p=\sum\limits_{i=1}^nu_{i,j_1}\cdots u_{i,j_{m}}\p.$$
It equals zero if $j_l\neq j_{l+1}$ for some $l$, otherwise
$$ \sum\limits_{i=1}^nu_{i,j_1}\cdots u_{i,j_{m}}\p=\sum\limits_{i=1}^nu_{i,j_1}^{m}\p=\sum\limits_{i=1}^nu_{i,j_1}^{m-1}\sum\limits_{l=1}^nu_{l,j_1}\p=\sum\limits_{i=1}^nu_{i,j_1}^{m-1}\p=\cdots=\sum\limits_{l=1}^nu_{l,j_1}\p=\p.$$
Since all partitions in $I_b(n)$ are concatenations of  blocks of arbitrary length, by Lemma \ref{concatenation}, the equation holds.
\end{proof}

The noncommutative distributional symmetries for Boolean independence are defined as follows.
\begin{definition}\normalfont  An orthogonal  Boolean quantum semigroup is a unital $C^*$-algebra $A$ generated by $n^2$ selfadjoint elements  $\{u_{i.j}| i,j=1,...,n\}$ and an orthogonal projection $\p$, such that the following hold:
\begin{itemize}
\item[1.]  $u=(u_{i,j})_{i,j=1,....n}\in  M_n(A)$ such that $\|u\|\leq 1$ and (u,$\p$) is $\p$-orthogonal.
\item[2.] $\Delta(u_{i,j})=\sum\limits_{k=1}^nu_{i,k}\otimes u_{k,j}$ and $\Delta{\bf P}={\bf P}\otimes {\bf P}, \Delta I=I\otimes I$ determine a $C^*$-unital homomorphism $\Delta:A\rightarrow A\otimes_{min}A$.
\end{itemize}
\end{definition}

\begin{definition}\normalfont
Let $(A,\Delta)$ be a quantum  semigroup and $\V$ be a unital algebra. By a right coaction of the quantum semigroup $A$ on $\V$, we mean a unital homomorphism $\alpha:\V\rightarrow\V\otimes \A$ such that
$$(\alpha\otimes id_A)\alpha=(id\otimes \Delta)\alpha.$$
\end{definition}

\begin{definition}\normalfont  Given an orthogonal  Boolean quantum semigroup $E(n)$ generated by $\{u_{i,j}\}_{i,j=1,...n}$ and $\p$, we have a natural coaction $\alpha_n$ of $E(n)$ on $\mathbb{C}\langle X_1,...,X_n \rangle$ such that 
$$\alpha_n: \mathbb{C}\langle X_1,...,X_n \rangle\rightarrow \mathbb{C}\langle X_1,...,X_n \rangle\otimes E(n) $$
is a  homomorphism defined via
$\alpha_n(X_i)=\sum_{k=1}^n X_k\otimes u_{k,i}$ for all $i$.
\end{definition}

\begin{definition} \normalfont Given a probability space $(\A,\phi)$, a sequence of random variables $(x_1,...,x_n)$ of $\A$ and an orthogonal  Boolean quantum semigroup $E(n)$ generated by $\{u_{i,j}\}_{i,j=1,...n}$ and $\p$. We  say that the joint distribution $\mu_{x_1,...,x_n}$ of $x_1,...,x_n$ is $E(n)$-invariant if 
$$ \mu_{x_1,...,x_n}(p)\p=\mu_{x_1,...,x_n}\otimes id_{E(n)}(\alpha_n(p))\p,$$
for all $p\in \C\langle X_1,...,X_n\rangle$.
\end{definition}

As orthogonal Hopf algebras,  we can define $E(n)$-invariance condition for an infinite sequence of random variables.  Given an orthogonal Boolean quantum semigroup $E(n)$ generated by $\{u_{i,j}\}_{i,j=1,...,n}$ and $\p$, then for $k\in\mathbb{N}$ , $E(n)$ can be considered as an orthogonal Boolean quantum semigroup $E(n,k)$ generated by $\{v_{i,j}\}_{i,j=1,...,n+k}$ and $\p'$ such that 
$$ v_{i,j}=
\left\{\begin{array}{ll}
u_{i,j} & \text{if}\,\, i,j \leq  n\\
\delta_{i,j}1_{E(n)} &\text{otherwise}
\end{array}\right.
$$
and $\p'=\p$.
We will call $E(n,k)$ the $k$-th extension of $E(n)$.
\begin{definition}\label{Boolean stationary}\normalfont
 Given a probability space $(\A,\phi)$, a sequence of random variables $(x_i)_{i\in\mathbb{N}}$ and an orthogonal Hopf algebra $E(n)$ generated by $\{u_{i,j}\}_{i,j=1,...n}$. We  say that the joint distribution $\mu$ of $(x_i)_{i\in\mathbb{N}}$ is $E(n)$-invariant if 
the joint distribution  of $(x_1,...,x_{n+k}) $ is $E(n,k)$-invariant for all $k\in\mathbb{N}$. 
\end{definition}

\begin{proposition}\label{4.5}\normalfont Let $(\A,\B,\E:\A\rightarrow\B)$ be an operator valued probability space and $\{x_i\}_{i=1,...,n}$  be a sequence of random variables in $\A$.  Let  $\phi$ be a linear functional on $\A$ such that $\phi(\cdot)=\phi(\E[\cdot])$.  Then, we have 
\begin{itemize}
\item If $(x_i)_{i=1,...,n}$ is  identically distributed  and Boolean independent with respect to $\E$, then the sequence is $B_s(n)$-invariant.
\item If $(x_i)_{i=1,...,n}$ is  identically symmetric distributed  and Boolean independent with respect to $\E$, then the sequence is $B_h(n)$-invariant.
\item If $(x_i)_{i=1,...,n}$ has  identically shifted Bernoulli distribution and  is Boolean independent with respect to $\E$, then the sequence is $B_b(n)$-invariant.
\item If $(x_i)_{i=1,...,n}$ has  identically centered Bernoulli distribution and Boolean independent with respect to $\E$, then the sequence is $B_o(n)$-invariant.
\end{itemize}
\end{proposition}

\begin{proof} 
Suppose that the joint distribution of $\{x_i\}_{i=1,...,n}$ satisfies one of the conditions specified in the statement  of the proposition, and let $D(k)$ be the partition family associated to the corresponding quantum semigroups. Let $X_{\jj}=X_{j_1}\cdots X_{j_k} $, by Lemma \ref{vanishing} and \ref{parititionanddistribution}, we have 

\begin{align*}
\mu_{x_1,...,x_n}(\alpha_n (X_{\jj}))\p&=\sum\limits_{\ii\in[n]^k}\mu_{x_1,...,x_n}(X_{\ii})u_{\ii,\jj}\p\\
&=\sum\limits_{\ii\in[n]^k}\phi(x_{\ii})u_{\ii,\jj}\p\\
&=\sum\limits_{\ii\in[n]^k}\phi(\E[x_{\ii}])u_{\ii,\jj}\p \\
&=\sum\limits_{\ii\in[n]^k}\sum\limits_{\pi\in D(k)}\phi(b_E^{(\pi)}(x_{\ii}))u_{\ii,\jj}\p\\
&=\sum\limits_{\pi\in D(k)}\sum\limits_{\ii\in[n]^k}\phi(b_E^{(\pi)}(x_{\ii}))u_{\ii,\jj}\p\\
&=\sum\limits_{\pi\in D(k)}\sum\limits_{\substack{ \ii\in[n]^k\\ \pi\leq \text{ker}\,\ii }}\phi(b_E^{(\pi)}(x_{\ii}))u_{\ii,\jj}\p\\
&=\sum\limits_{\pi\in D(k)}\sum\limits_{\substack{ \ii\in[n]^k\\ \pi\leq \text{ker}\,\ii }}\phi(b_E^{(\pi)}(x_1,...,x_1))u_{\ii,\jj}\p\\
&=\sum\limits_{\substack{ \pi\in D(k)\\ \pi\leq \text{ker}\,\jj }}\phi(b_E^{(\pi)}(x_1,...,x_1))\p\\
&=\sum\limits_{\substack{ \pi\in D(k)\\ \pi\leq \text{ker}\,\jj }}\phi(b_E^{(\pi)}(x_{\jj}))\p\\
&=\phi(\E[x_{\jj}])\p\\
&=\phi(x_{\jj})p\\
&=\mu_{x_1,...,x_n}(X_{\jj})\p,
\end{align*}
which completes the proof.
\end{proof}

\section{Main result}
In this section, we will prove Theorem \ref{Classical} and Theorem \ref{Main theorem}.  
Then, we present an application of these theorems to easy groups $C_{s'}(n)$ and $C_{b'}(n)$, easy quantum groups $A_{s'}(n)$, $A_{b'}(n)$ and $ A_{b^{\#}}(n)$, Boolean quantum semigroups $B_{s'}(n)$ and $B_{b'}(n)$. 
We start with free independence, since the proof of the other parts can be easily derived from this case.
 
Given a $W^*$-probability space $(\A,\phi)$ such that $\phi$ is faithful. 
 Let $\{E(n)\}_{n\in\mathbb{N}}$ be a sequence of orthogonal Hopf algebras such that $A_s(n)\subseteq E(n)\subseteq A_o(n)$ for each $n$. 
 Let $(x_i)_{i\in\mathbb{N}}$ be a sequence of random variables which generate $A$. 
 Suppose that  the joint distribution of $(x_i)_{i\in\mathbb{N}}$ is $E(n)$-invariant for all $n$.  
 By Proposition \ref{Implication}, $(x_i)_{i\in\mathbb{N}}$ are $A_s(n)$-invariant for all $n$.  
 By K\"ostler and Speicher \cite{KS}, there exist a $W^*$-subalgebra $\B$ such that $1\subseteq \B\subseteq \A$  and a $\phi$-preserving conditional expectation $\E:\A\rightarrow \B$ such that $(x_i)_{i\in\mathbb{N}}$ are freely independent and identically distributed  with respect to $\E$.  
This is exactly  the statement 1 for the  free case. 
In addition, by Proposition 4.3 in \cite{KS} and Definition \ref{stationary}, the coaction invariant condition for $\phi$ can be extended to the conditional expectation $\E$, i.e. 
$$ \E[b_0 x_{i_1}b_{1}\cdots b_{k-1}x_{i_k}b_{k}]\otimes 1_{E(n)}= \sum\limits_{j_1,...,j_k=1}^n\E[b_0 x_{j_1}b_{1}\cdots b_{k-1}x_{j_k}b_{k}]\otimes u_{j_1,i_1}\cdots u_{j_k,i_k} $$
for $i_1,...,i_k\leq n$, where $\{u_{i,j}|i,j=1,...,n\}$ are generators of $E(n)$.

2.  Suppose that  $A_s(n) \subseteq E(n) \subseteq A_b(n)$ for all $n$ and there exists a number $k$ such that $E(k)\neq A_s(k)$.  Let $\{u_{i,j}\}_{i,j=1,...,k}$ be generators of $E(k)$. By proposition \ref{3.4}, there exists an index $i'$ such that
 $$\sum\limits_{l=1}^k u^m_{l,i'}\neq 1$$
 for all $m>2$.
 Without loss of generality, we assume that $i'=1$. It is sufficient to show that  $\kappa_l(x_1b_1,....,x_1b_l)=0$ for all $l\geq 3$, where $b_1,....,b_l\in B$.  We check it by induction on $l$. 
First,  we have that 
\begin{align*}
&\E[x_1 b_1\cdots x_1 b_l]\otimes 1_{E(n)}\\
=&\sum\limits_{\ii\in[k]^l}\E[x_{i_1}b_1\cdots x_{i_l}b_l]\otimes u_{\ii,1} \\
=&\sum\limits_{\ii\in[k]^l}\sum\limits_{\pi\in NC(l)}\kappa_\pi(x_{i_1}b_1,...,x_{i_l}b_l)\otimes u_{\ii,1}\\
=&\sum\limits_{\pi\in NC_b(l)}\sum\limits_{\ii\in[k]^l}\kappa_\pi(x_{i_1}b_1,...,x_{i_l}b_l)\otimes u_{\ii,1}+\sum\limits_{\pi\in NC(l)\setminus NC_b(l)}\sum\limits_{\ii\in[k]^l}\kappa_\pi(x_{i_1}b_1,...,x_{i_l}b_l)\otimes u_{\ii,1}\\
=&\sum\limits_{\pi\in NC_b(l)}\sum\limits_{\substack{ \ii\in[k]^l\\ \pi\leq \text{ker}\,\ii }}\kappa_\pi(x_{i_1}b_1,...,x_{i_l}b_l)\otimes u_{\ii,1}+\sum\limits_{\pi\in NC(l)\setminus NC_b(l)}\sum\limits_{\substack{ \ii\in[k]^l\\ \pi\leq \text{ker}\,\ii }}\kappa_\pi(x_{i_1}b_1,...,x_{i_l}b_l)\otimes u_{\ii,1}\\
=&\sum\limits_{\pi\in NC_b(l)}\sum\limits_{\substack{ \ii\in[k]^l\\ \pi\leq \text{ker}\,\ii }}\kappa_\pi(x_1b_1,...,x_1b_l)\otimes u_{\ii,1}+\sum\limits_{\pi\in NC(l)\setminus NC_b(l)}\sum\limits_{\substack{ \ii\in[k]^l\\ \pi\leq \text{ker}\,\ii }}\kappa_\pi(x_{1}b_1,...,x_{1}b_l)\otimes u_{\ii,1}\\
=&\sum\limits_{ \pi\in NC_b(l)}\kappa_\pi(x_1b_1,...,x_1b_l)\otimes 1_{E(n)}+\sum\limits_{\pi\in NC(l)\setminus NC_b(l)}\sum\limits_{\substack{ \ii\in[k]^l\\ \pi\leq \text{ker}\,\ii }}\kappa_\pi(x_{1}b_1,...,x_{1}b_l)\otimes u_{\ii,1}.
\end{align*}

The first term of the last equality follows because $E(n)$ is a quotient algebra of $A_b(n)$.
On the other hand, we have that 
$$\E[x_1 b_1,...,x_1 b_l]\otimes 1_{E(n)}= \sum\limits_{ \pi\in NC_b(k)}\kappa_\pi(x_1b_1,...,x_1b_l)\otimes 1_{E(n)}+\sum\limits_{\pi\in NC(l)\setminus NC_b(l)}\kappa_\pi(x_1b_1,...,x_1b_l)\otimes 1_{E(n)}.$$ 
It follows that
\begin{equation}\label{2}
\sum\limits_{\pi\in NC(l)\setminus NC_b(l)}\sum\limits_{\substack{ \ii\in[k]^l\\ \pi\leq \text{ker}\,\ii }}\kappa_\pi(x_{1}b_1,...,x_{1}b_l)\otimes u_{\ii,1}=\sum\limits_{\pi\in NC(l)\setminus NC_b(l)}\kappa_\pi(x_1b_1,...,x_1b_l)\otimes 1_{E(n)}.
\end{equation}
When $l=3$, we have $NC(3)\setminus NC_b(3)=\{1_3\}$. Then
$$\sum\limits_{\substack{ \ii\in[n]^k\\ \pi\leq \text{ker}\,1_3}}\kappa_{1_3}(x_{1}b_1,...,x_{1}b_3)\otimes u_{\ii,1}=\kappa_{1_3}(x_1b_1,...,x_1b_3)\otimes 1_{E(n)},$$
which is
$$\kappa_{1_3}(x_{1}b_1,...,x_{1}b_k)\otimes (\sum\limits_{l=1}^{k}u^3_{l,1}-1_{E(n)})=0.$$
Therefore, $\kappa_{1_3}(x_{1}b_1,...,x_{1}b_3)=0.$
Suppose $\kappa_{1_l}(x_{1}b_1,...,x_{1}b_l)=0$ for $3\leq l\leq m$, then for $\pi\in NC(m+1)$, $\kappa_\pi(x_{i_1}b_1,...,x_{1}b_{m+1})=0$ if $\pi$ contains a block whose size is between $3$ and $m$.  Each partition $\pi\in NC(m+1)\setminus NC_b(m+1)$ contains at least one block whose size is greater than $2$. Therefore, for $\pi\in NC(m+1)\setminus NC_b(m+1)$, $\kappa_{\pi}(x_{1}b_1,...,x_{1}b_k)=0$ if $\pi\neq 1_{m+1}$. Hence, equation \ref{2} becomes
$$\kappa_{1_{m+1}}(x_{1}b_1,...,x_{1}b_{m+1})\otimes (\sum\limits_{l=1}^{k}u^{m+1}_{l,1}-1_{E(n)})=0$$
which implies 
$$ \kappa_{1_{m+1}}(x_{1}b_1,...,x_{1}b_{m+1})=0,$$
for all $b_1,...,b_{m+1}\in \B$. The proof is complete.

3.  
Suppose that  $A_s(n) \subseteq E(n) \subseteq A_h(n)$ for all $n$ and there exists a $k$ such that $E(k)\neq A_s(k)$.  Let $\{u_{i,j}\}_{i,j=1,...,k}$ be generators of $E(k)$. By proposition \ref{3.4}, $\exists\, i'$ such that 
 $$\sum\limits_{l=1}^k u^m_{l,i'}\neq 1,$$
 for all odd numbers $m$.  Without loss of  generality, we assume that $i'=1$.
We need to  show that  $\kappa_l(x_1b_1,....,x_1b_l)=0$ for all add numbers $l$ where $b_1,....,b_l\in B$.   We prove this by induction on $l$. 
We have that 
$$\begin{array}{rcl}
&&\E[x_1 b_1\cdots x_1 b_l]\otimes 1_{E(n)}\\
&=&\sum\limits_{\pi\in NC_h(l)}\sum\limits_{\substack{ \ii\in[k]^l\\ \pi\leq \text{ker}\,\ii }}\kappa_\pi(x_1b_1,...,x_1b_l)\otimes u_{\ii,1}+\sum\limits_{\pi\in NC(l)\setminus NC_h(l)}\sum\limits_{\substack{ \ii\in[k]^l\\ \pi\leq \text{ker}\,\ii }}\kappa_\pi(x_{1}b_1,...,x_{1}b_l)\otimes u_{\ii,1}\\
&=&\sum\limits_{ \pi\in NC_h(l)}\kappa_\pi(x_1b_1,...,x_1b_l)\otimes 1_{E(n)}+\sum\limits_{\pi\in NC(l)\setminus NC_h(l)}\sum\limits_{\substack{ \ii\in[k]^l\\ \pi\leq \text{ker}\,\ii }}\kappa_\pi(x_{1}b_1,...,x_{1}b_l)\otimes u_{\ii,1}.
\end{array}
$$
The first term of the last equality follows because $E(n)$ is a quotient algebra of $A_h(n)$. On the other hand, we have
$$\E[x_1 b_1,...,x_1 b_l]\otimes 1_{E(n)}= \sum\limits_{ \pi\in NC_h(l)}\kappa_\pi(x_1b_1,...,x_1b_l)\otimes 1_{E(n)}+\sum\limits_{\pi\in NC(l)\setminus NC_h(l)}\kappa_\pi(x_1b_1,...,x_1b_l)\otimes 1_{E(n)}.$$ 
Therefore, 
\begin{equation}\label{3}
\sum\limits_{\pi\in NC(l)\setminus NC_h(l)}\sum\limits_{\substack{ \ii\in[k]^l\\ \pi\leq \text{ker}\,\ii }}\kappa_\pi(x_{1}b_1,...,x_{1}b_l)\otimes u_{\ii,1}=\sum\limits_{\pi\in NC(l)\setminus NC_h(l)}\kappa_\pi(x_1b_1,...,x_1b_l)\otimes 1_{E(n)}.
\end{equation}
When $l=1$, we have $NC(1)\setminus NC_h(1)=\{1_1\}$, then
$$\kappa^{(1)}(x_{1}b_1)\otimes (\sum\limits_{l=1}^{k}u_{l,1}-1_{E(n)})=0.$$
Therefore, $\kappa_{1_1}(x_{1}b_1)=0.$
Suppose $\kappa_{1_l}(x_{1}b_1,...,x_{1}b_l)=0$ for odd numbers $ l\leq 2m$, then for $\pi\in NC(2m+1)$, $\kappa_\pi(x_{i_1}b_1,...,x_{1}b_{2m+1})=0$ if $\pi$ contains a block whose size is an odd number less than  $2m$.   Each partition $\pi\in NC(2m+1)$ contains at least one block whose size is odd. Therefore, for $\pi\in NC(2m+1)$, $\kappa_{\pi}(x_{1}b_1,...,x_{1}b_{2m+1})=0$ if $\pi\neq 1_{2m+1}$. Hence, equation \ref{3} becomes
$$\kappa_{1_{2m+1}}(x_{1}b_1,...,x_{1}b_{2m+1})\otimes (\sum\limits_{l=1}^{k}u^{2m+1}_{l,1}-1_{E(n)})=0$$
which implies 
$$ \kappa_{1_{2m+1}}(x_{1}b_1,...,x_{1}b_{2m+1})=0,$$
for all $b_1,...,b_{m+1}\in \B$. The proof is complete.

4. If  there exist  $k_1, k_2$ such that $E(k_1)\not\subseteq A_h(k_1)$  and $E(k_2)\not\subseteq A_b(k_2)$, by Case 2 and Case 3, the only non-vanishing cumulants are pair partition cumulants.  The proof is done.

The proof of Theorem \ref{Classical}  follows the  free case replacing noncrossing partitions by arbitrary partitions.  The existence of the conditional expectation is a well known result in classical probability, see \cite{Ka}.

The proof of the Boolean case of Theorem \ref{Main theorem} is slightly different.   
For Boolean de Finetti type theorem, we need to consider random variables in a $W^*$-probability space with a non-degenerated state $\phi$ \cite{Liu}. 
 In addition, we assume that $\A$ is generated by a sequence of random variables $(x_i)_{i\in\mathbb{N}}$. 
 Let $\{E(n)\}_{n\in\mathbb{N}}$ be a sequence of orthogonal Boolean quantum groups such that $B_s(n)\subseteq E(n)\subseteq B_o(n)$ for each $n$. 
 If the joint distribution of $(x_i)_{i\in\mathbb{N}}$ is $E(n)$-invariant, then the joint distribution of $(x_i)_{i\in\mathbb{N}}$ is $B_s(n)$ invariant for all $n$. 
 By the main result in \cite{Liu}, there are a $W^*$-subalgebra $\B$(not necessarily contain the unit of $\A$) such that $\B\subseteq \A$ and a $\phi$-preserving conditional expectation $\E:\A\rightarrow \B$ such that $(x_i)_{i\in\mathbb{N}}$ are Boolean independent and identically distributed  with respect to $\E$. 
 In this part of proof, we assume that $\B$ does not contain $1_{\A}$. 
Then, the tail algebra
$$\B=\bigcap\limits_{n=1}^\infty W^*\{x_k|k\geq n\},$$
where $W^*\{x_k|k\geq n\}$ is the WOT closure of the non-unital  algebra generated by $\{x_k|k\geq n\}$. 
We call $\B$ the non-unital tail algebra of $\{x_i\}_{i\in\mathbb{N}}$.
Unlike the proof for free independence, the coaction invariant conditions for $\phi$ cannot be extended to the conditional expectation $E$ directly.  Actually, we have a stronger statement.

\begin{proposition}\label{Boolean property} \normalfont
Let $(\A,\phi)$ be a $W^*$-probability space and $(x_i)_{i\in\mathbb{N}}$ be a sequence of selfadjoint random variables which generate $\A$ as a von Neumann algebra and the unit of $\A$ is  contained in the WOT closure of the non-unital algebra generated by $(x_i)_{i\in\mathbb{N}}$ . 
Let $E(n)$ be a sequence of Boolean orthogonal quantum semigroups such that $B_s(n)\subseteq E(n)\subseteq B_o(n)$. 
If $(x_i)_{i\in\mathbb{N}}$ are $E(n)$-invariant for all $n$, then there exists a $\phi$-preserving conditional expectation $\E:\A\rightarrow \B$, where $\B$ is the non-unital tail algebra of $\{x_i\}_{i\in\mathbb{N}}$, such that $(x_i)_{i\in\mathbb{N}}$ is Boolean independent with respect to $\E$. Let $\A_n$ be the non-unital algebra generated by $\{x_i\}_{i\in\mathbb{N}}$. 
We have that 
$$\E[a_1ba_2]=\E[a_1]b\E[a_2],$$
where $ a_1, a_2 \in\A_n$ for some $n$ and $b\in\B$.
 Let $\{u_{i,j}\}_{i,j=1,...,n}$ and $\p$ be the generators of $E(n)$.
 We  have that 
$$ \E[ x_{i_1}\cdots x_{i_k}]\otimes \p= \sum\limits_{j_1,...,j_k=1}^n\E[x_{j_1}\cdots x_{j_k}]\otimes u_{j_1,i_1}\cdots u_{j_k,i_k} \p,$$
for $i_1,...,i_k\leq n$.
\end{proposition}
\begin{proof}
The existence of $\E$ is shown in \cite{Liu}. We turn to prove the two equations of this proposition.
Given $ a_1, a_2 \in\A_n$ for some $n$ and $b\in\B$, by assumption, $b$ is contained in the $W^*$-closure of the non-unital algebra generated by $\{x_i|i>n\}$. 
By Kaplansky theorem, there exist a net of bounded elements $(y_i)_{i\in\mathbb{\omega}}$ such that $(y_i)_{i\in\mathbb{\omega}}$ are contained in  the non-unital algebra generated by $\{x_i|i>n\}$ and  converge to $b$ in strong operator topology. 
Therefore, by normality of $\E$, we have
$$\E[a_1ba_2]=\lim\limits_{i\rightarrow\infty}\E[a_1y_ia_2]=\lim\limits_{i\rightarrow\infty}\E[a_1]\E[y_i]\E[a_2]=\E[a_1]b\E[a_2] ,$$
where the second equality follows because $(x_i)_{i\in\mathbb{N}}$  are Boolean independent with respect to $\E$.
The second equation of this proposition can be checked pointwisely.  
Let $a_1,a_2\in \A_{m}$ for some $m$.  
In \cite{Liu}, it was shown that there exists a normal homomorphism $\alpha:\A\rightarrow\A$ such that $\alpha(x_i)=x_{i+1}$ for all $i\in\mathbb{N}$.  By the proof of Lemma 6.7 in \cite{Liu} and the assumption that $\{x_i\}_{i\in\mathbb{N}}$ is $E(n)$-invariant, we have 
\begin{align*}
&\phi (a_1\E[ x_{i_1}\cdots x_{i_k}] a_2)\otimes \p\\
=&\lim\limits_{l\rightarrow \infty,l>m}\phi (a_1\alpha^l( x_{i_1}\cdots x_{i_k}) a_2)\otimes \p\\
=&\lim\limits_{l\rightarrow \infty,l>m}\phi ( \alpha^n(a_1) x_{i_1}\cdots x_{i_k}  \alpha^n(a_2))\otimes \p\\
=& \lim\limits_{l\rightarrow \infty,l>m}(\phi( \alpha^n (a_1)\sum\limits_{j_1,...,j_k=1}^n x_{j_1}\cdots x_{j_k}\alpha^n(a_2))\otimes u_{j_1,i_1}\cdots u_{j_k,i_k} \p\\
=& \lim\limits_{l\rightarrow \infty,l>m}\phi(a_1\alpha^l(\sum\limits_{j_1,...,j_k=1}^n x_{j_1}\cdots x_{j_k})a_2)\otimes u_{j_1,i_1}\cdots u_{j_k,i_k} \p\\
=& \sum\limits_{j_1,...,j_k=1}^n\phi(a_1\E[x_{j_1}\cdots x_{j_k}]a_2)\otimes u_{j_1,i_1}\cdots u_{j_k,i_k} \p.\\
\end{align*}
Since $a_1,a_2$ are arbitrarily chosen from the dense set $\bigcup\limits_{n\rightarrow\infty}\A_n$ of $\A$, the proof is done.
\end{proof}

Now, we are ready to prove  Boolean case of Theorem \ref{Main theorem}\\
\begin{proof}
Statement 1.  This is  the Boolean de Finetti theorem in \cite{Liu}.

Statement 2.   As  the free case, we need to show that $b_\E^{(l)}(x_1b_1,....,x_1b_l)=0$ for all $l\geq 3$ where $b_1,....,b_l\in B\cup\{\mathbb{C}1_{\A}\}$. By proposition \ref{Boolean property}, we have 
$$\begin{array}{rcl}

&&\E[x_{\ii_1}b_1x_{\ii_2}\cdots b_{n-1}x_{\ii_m}]\\
&=&\E[x_{\ii_1}]b_1 \E[x_{\ii_2}]\cdots b_{n-1}\E[x_{\ii_m}]\\
&=&\sum\limits_{\pi_1\in I(k_1)}b_\E^{(\pi_1)}(x_{i^{(1)}_1},...x_{i^{(1)}_{k_1}}) b_1 \sum\limits_{\pi_2\in I(k_2)}b_\E^{(\pi_2)}(x_{i^{(2)}_1},...x_{i^{(2)}_{k_2}})\cdots b_{n-1}\sum\limits_{\pi_m\in I(k_m)}b_\E^{(\pi_m)}(x_{i^{(m)}_1},...x_{i^{(m)}_{k_m}})\\
&=&\sum\limits_{\pi\in I(k_1)\times I(k_2)\times\cdots \times I(k_m)}b_\E^{(\pi)}(x_{i^{(1)}_1},...x_{i^{(1)}_{k_1}},b_1x_{i^{(2)}_1},...x_{i^{(2)}_{k_2}},\cdots,b_{n-1}x_{i^{(m)}_1},...x_{i^{(m)}_{k_m}}),\\
\end{array}$$
where $\ii_l=(i^{(l)}_1,...,i^{(l)}_{k_l})\in[n]^{k_l}$ for all $l=1,...,m$ for some $n$ and $b_1,...,b_m\in\B$.  Therefore, to finish the proof, we need to show that $b_\E^{(k)}(x_1,....,x_1)=0$ for all $l\geq 3$.   The rest of the poof is similar to  the free case:  \\
Let $\{u_{i,j}\}_{i,j=1,...,k}$'s and $\p$ be the generators of $E(k)$. First,  by Proposition \ref{Boolean property}, we have 
\begin{align*}
&\E[x_1\cdots x_1]\otimes \p\\
=&\sum\limits_{\ii\in[k]^l}\E[x_\ii]\otimes u_{\ii,1}\p \\
=&\sum\limits_{\ii\in[k]^l}\sum\limits_{\pi\in I(l)}b_\E^{(\pi)}(x_\ii)\otimes u_{\ii,1}\\
=&\sum\limits_{\pi\in I_b(l)}\sum\limits_{\ii\in[k]^l}b_\E^{(\pi)}(x_{i_1},...,x_{i_l})\otimes u_{\ii,1}\p+\sum\limits_{\pi\in I(l)\setminus I_b(l)}\sum\limits_{\ii\in[k]^l}b_\E^{(\pi)}(x_{i_1},...,x_{i_l})\otimes u_{\ii,1}\p\\
=&\sum\limits_{\pi\in I_b(l)}\sum\limits_{\substack{ \ii\in[k]^l\\ \pi\leq \text{ker}\,\ii }}b_\E^{(\pi)}(x_{i_1},...,x_{i_l})\otimes u_{\ii,1}\p+\sum\limits_{\pi\in I(l)\setminus I_b(l)}\sum\limits_{\substack{ \ii\in[k]^l\\ \pi\leq \text{ker}\,\ii }}b_\E^{(\pi)}(x_{i_1},...,x_{i_l})\otimes u_{\ii,1}\p\\
=&\sum\limits_{\pi\in I_b(l)}\sum\limits_{\substack{ \ii\in[k]^l\\ \pi\leq \text{ker}\,\ii }}b_\E^{(\pi)}(x_1,...,x_1)\otimes u_{\ii,1}\p+\sum\limits_{\pi\in I(l)\setminus I_b(l)}\sum\limits_{\substack{ \ii\in[k]^l\\ \pi\leq \text{ker}\,\ii }}b_\E^{(\pi)}(x_{1},...,x_{1})\otimes u_{\ii,1}\p\\
=&\sum\limits_{ \pi\in I_b(l)}b_\E^{(\pi)}(x_1b_1,...,x_1b_l)\otimes \p+\sum\limits_{\pi\in I(l)\setminus I_b(l)}\sum\limits_{\substack{ \ii\in[k]^l\\ \pi\leq \text{ker}\,\ii }}b_\E^{(\pi)}(x_{1},...,x_{1})\otimes u_{\ii,1}\p.\\
\end{align*}

The first term of the last equality follows because $E(n)$ is a quotient algebra of $B_b(n)$.
On the other hand
$$\E[x_1 ,...,x_1 ]\otimes \p= \sum\limits_{ \pi\in I_b(k)}b_\E^{(\pi)}(x_1,...,x_1)\otimes \p+\sum\limits_{\pi\in I(l)\setminus I_b(l)}b_\E^{(\pi)}(x_1,...,x_1)\otimes \p.$$ 
Therefore, 
\begin{equation}\label{4}
\sum\limits_{\pi\in I(l)\setminus I_b(l)}\sum\limits_{\substack{ \ii\in[k]^l\\ \pi\leq \text{ker}\,\ii }}b_\E^{(\pi)}(x_1,...,x_1)\otimes u_{\ii,1}\p=\sum\limits_{\pi\in I(l)\setminus I_b(l)}b_\E^{(\pi)}(x_1,...,x_1)\otimes \p.
\end{equation}
By assumption, $E(k)$ has a quotient algebra $E'(k)$ such that $A_s(k)\not\subseteq E'(k)\subseteq A_n(n)$. 
Let $\{u'_{i,j}\}$ be the generators of $E'(k)$. Then, there exists a $C^*$-homomorphism $\Psi:E(k)\rightarrow E'(k)$ such that 
$$ \Psi(u_{i,j})=u_{i,j}' \,\,\, \text{for all}\,\, i,j=1,...,k,\,\,\,\,\text{and}\,\,\,\Psi(\p)=1_{E'(k)}.$$ 
Without loss of generality, by proposition \ref{3.4}, we can assume that  
 $$\sum\limits_{l=1}^k (u'_{l,1})^m\neq 1,$$
 for all $m>2$.
 Let $id\otimes \Psi$ acts on equation \ref{4}. Then, we get

\begin{equation}\label{4}
\sum\limits_{\pi\in I(l)\setminus I_b(l)}\sum\limits_{\substack{ \ii\in[k]^l\\ \pi\leq \text{ker}\,\ii }}b_\E^{(\pi)}(x_{1},...,x_{1})\otimes u'_{\ii,1}=\sum\limits_{\pi\in I(l)\setminus I_b(l)}b_\E^{(\pi)}(x_1,...,x_1)\otimes 1_{E'(k)}.
\end{equation}

When $l=3$, we have $I(3)\setminus I_b(3)=\{1_3\}$, then
$$\sum\limits_{\substack{ \ii\in[n]^k\\ \pi\leq \text{ker}\,1_3}}b_\E^{(3)}(x_{1},...,x_{1})\otimes u'_{\ii,1}=b_\E^{(3)}(x_1,...,x_1)\otimes 1_{E'(k)},$$
which is
$$\kappa_{1_3}(x_{1},...,x_{1})\otimes (\sum\limits_{l=1}^{k}u'^3_{l,1}-1_{E'(k)})=0.$$
Therefore, $b_\E^{(3)}(x_{1},...,x_{1})=0.$
Suppose $b_\E^{(l)}(x_{1}b_1,...,x_{1}b_l)=0$ for $3\leq l\leq m$. 
Then, for $\pi\in I(m+1)$, $b_E^{(\pi)}(x_1,...,x_{1})=0$ if $\pi$ contains a block whose size is between $3$ and $m$.  Each partition $\pi\in I(m+1)\setminus I_b(m+1)$ contains at least one block whose size is greater than $2$. Therefore, for $\pi\in I(m+1)\setminus I_b(m+1)$, $b_E^{(\pi)}(x_1,...,x_{1})=0$ if $\pi\neq 1_{m+1}$. Hence, equation \ref{2} becomes
$$b_E^{(m+1)}(x_{1},...,x_{1})\otimes (\sum\limits_{l=1}^{k}u'^{m+1}_{l,1}-1_{E'(k)})=0$$
which implies 
$$ b_E^{(m+1)}(x_{1},...,x_{1})=0.$$ The proof of Statement 2 is complete.

Similarly, by the proof  of Statement 3 and Statement 4 for free case,  Statement 3 and Statement 4 for Boolean independence are also true.
\end{proof}

\begin{remark}\normalfont
 Our general de Finetti theorem for Boolean independence is not complete since we know very little about the classification of Boolean quantum semigroups. 
\end{remark}

According the diagrams in Section 3 and Section 4,  we have  the following relations.
\begin{itemize}
\item[1.]$C_s(n)\subseteq C_{s'}(n)\subseteq C_b(n) $ 
for all n,  and $C_s(n)\neq C_{s'}(n)$ for $n>3$.
\item[2.] $C_{b'}(n)\not\subseteq C_{h}(n)$ and $C_{b'}(n)\not\subseteq C_b(n)$ for $n>3$.
\item[3.] $A_s{n}\subseteq A_{s'}(n)\subseteq A_b(n) $ 
for all n,  and $A_s(n)\neq A_{s'}(n)$ for $n>3$.
\item[4.] $A_{b'}(n),A_{b^\#}(n)\not\subseteq A_{h}(n)$ and $A_{b'}(n),A_{b^\#}(n)\not\subseteq A_b(n)$  for $n>3$.
\item[5.] $B_s(n)\subseteq B_{s'}(n)\subseteq B_b(n)$
for all n,  and $B_s(n)\neq B_{s'}(n)$ for $n>3$. Moreover $A_{s'}(n)$ is a quotient algebra of $B_{s'(n)}$
\item[6.] $A_{b'}$ is a quotient algebra of $B_{b'}(n)$  for $n>3$.
\end{itemize}

By Theorem\ref{Classical}, we have the following corollary for classical independence.
\begin{corollary}\normalfont
Let $(\Omega,\Sigma, \mu)$ be a classical probability space and  $(x_i)_{i\in\mathbb{N}}$ be a sequence of real-valued  random variables such that $x_i\in\bigcap\limits_{1\leq p<\infty} L^{p}(\Omega,d\mu)$ for all $i$.
 Let $\A$ be the algebra of all random variables and $\phi$ is the expectation. Assume that $(x_i)_{i\in\mathbb{N}}$ generate $\A$.
\begin{itemize}
\item[1.]   If the joint distribution of $(x_i)_{i\in\mathbb{N}}$ is $C_{s'}(n)$ invariant for all $n\in\mathbb{N}$, then there exist a subalgebra $\B$  and a $\phi$-preserving conditional expectation $\E:\A\rightarrow \B$ such that  such that $(x_i)_{i\in\mathbb{N}}$ are conditionally independent and have  identically symmetric distribution  with respect to $\E$.
\item[2.]  If the joint distribution of $(x_i)_{i\in\mathbb{N}}$ is $C_{b'}(n)$ invariant for all $n\in\mathbb{N}$, then there exist a subalgebra $ \B$  and a $\phi$-preserving conditional expectation $\E:\A\rightarrow \B$ such that such that $(x_i)_{i\in\mathbb{N}}$ are conditionally independent and have  centered Gaussian  distribution  with respect to $\E$.
\end{itemize}
\end{corollary}

By Theorem \ref{Main theorem}, we have the following result for free and Boolean independence.

\begin{corollary}\normalfont
Let $(\A, \phi)$ be a $W^*$-probability space and $(x_i)_{i\in\mathbb{N}}$ be a sequence of random variables which generate $\A$. 
\begin{itemize}
\item Free case:   Suppose that $\phi$ is faithful.  
\begin{itemize}
\item[1.]  If the joint distribution of $(x_i)_{i\in\mathbb{N}}$ is $A_{s'}(n)$ invariant for all $n\in\mathbb{N}$, then there exist a $W^*$-subalgebra $1\subseteq \B\subseteq \A$  and a $\phi$-preserving conditional expectation $\E:\A\rightarrow \B$ such that $(x_i)_{i\in\mathbb{N}}$ are freely independent and have  identically symmetric distribution  with respect to $\E$.
\item[2.]   If the joint distribution of $(x_i)_{i\in\mathbb{N}}$ is $A_{b'}(n)$-invariant for all $n\in\mathbb{N}$, then exist  a $W^*$-subalgebra $1\subseteq \B\subseteq \A$  and a $\phi$-preserving conditional expectation $\E:\A\rightarrow \B$ such that $(x_i)_{i\in\mathbb{N}}$ are freely independent and have  centered semicircular  distribution  with respect to $\E$.
\item[3.]   If the joint distribution of $(x_i)_{i\in\mathbb{N}}$ is $A_{b^\#}(n)$-invariant for all $n\in\mathbb{N}$, then exist  a $W^*$-subalgebra $1\subseteq \B\subseteq \A$  and a $\phi$-preserving conditional expectation $\E:\A\rightarrow \B$ such that $(x_i)_{i\in\mathbb{N}}$ are freely independent and have  centered semicircular  distribution  with respect to $\E$.
\end{itemize}

\item Boolean case: 
If $\phi$ is non-degenerated. 
\begin{itemize}
\item[1.]  If the joint distribution of $(x_i)_{i\in\mathbb{N}}$ is $B_{s'}(n)$ invariant for all $n\in\mathbb{N}$, then there are a $W^*$-subalgebra(not necessarily contain the unit of $\A$) $\B\subseteq \A$ and a $\phi$-preserving conditional expectation $\E:\A\rightarrow \B$ such that $(x_i)_{i\in\mathbb{N}}$ are Boolean independent and have  identically symmetric distribution  with respect to $\E$.

\item[2.]  If the joint distribution of $(x_i)_{i\in\mathbb{N}}$ is $B_{b'}(n)$ invariant for all $n\in\mathbb{N}$, then there are a $W^*$-subalgebra(not necessarily contain the unit of $\A$) $\B\subseteq \A$ and a $\phi$-preserving conditional expectation $\E:\A\rightarrow \B$ such that $(x_i)_{i\in\mathbb{N}}$ are conditionally independent and have  centered Bernoulli distribution  with respect to $\E$.
\end{itemize}

\end{itemize}
\end{corollary}

\noindent{\bf Acknowledgment} The author would like to thank his thesis advisor, D.-V. Voiculescu, for his continued guidance and support while completing this project.While working on this paper, the author was supported in part by funds from NSF grant DMS-1301727.

\bibliographystyle{plain}

\bibliography{references}

\vspace{1cm}

\noindent Department of Mathematics\\
Indiana University	\\
Bloomington, IN 47401, USA\\
E-MAIL: liuweih@indiana.edu \\

\end{document}